\documentclass[12pt]{article}
\usepackage{amssymb}
\usepackage{amsmath}
\usepackage{amsthm}
\usepackage{enumerate}
\usepackage{color}
\usepackage{xcolor}
\usepackage{fullpage}
\usepackage{hyperref}
\usepackage{comment}

\newtheorem{theorem}{Theorem}[section]

\newtheorem{defn}[theorem]{Definition}

\newtheorem{corollary}[theorem]{Corollary}
\newtheorem{lemma}[theorem]{Lemma}

\newtheorem*{claim*}{Claim}
\newtheorem{conjecture}[theorem]{Conjecture}

\newtheorem{remark}[theorem]{Remark}

\newtheorem{theorem LSV}{Theorem LSV}
\newtheorem*{theorem LSV*}{Theorem LSV}
\newtheorem*{theorem*}{Theorem}

\newtheorem{theorem MTP}{Mass Transference Principle}
\newtheorem*{theorem MTP*}{Mass Transference Principle}
\newtheorem{theorem K}{Khintchine's Theorem}
\newtheorem*{theorem K*}{Khintchine's Theorem}
\newtheorem{theorem J}{Jarn\'{\i}k's Theorem}
\newtheorem*{theorem J*}{Jarn\'{\i}k's Theorem}
\newtheorem{theorem KJ}{Khintchine--Jarn\'{\i}k Theorem}
\newtheorem*{theorem KJ*}{Khintchine--Jarn\'{\i}k Theorem}
\newtheorem{theorem BV1}{Theorem BV1}
\newtheorem*{theorem BV1*}{Theorem BV1}
\newtheorem{theorem BV2}{Theorem BV2}
\newtheorem*{theorem BV2*}{Theorem BV2}
\newtheorem{theorem KG}{Theorem KG}
\newtheorem*{theorem KG*}{Theorem KG}
\newtheorem{theorem IHKG}{Inhomogeneous Khintchine--Groshev Theorem}
\newtheorem*{theorem IHKG*}{Inhomogeneous Khintchine--Groshev Theorem}
\newtheorem{theorem DLN1}{Theorem DLN1}
\newtheorem*{theorem DLN1*}{Theorem DLN1}
\newtheorem{theorem DLN2}{Theorem DLN2}
\newtheorem*{theorem DLN2*}{Theorem DLN2}
\newtheorem{theorem DLN3}{Theorem DLN3}
\newtheorem*{theorem DLN3*}{Theorem DLN3}
\newtheorem{theorem S}{Theorem S}
\newtheorem*{theorem S*}{Theorem S}
\newtheorem*{theorem L*}{Previous results for homogeneous metric Diophantine approximation}
\newtheorem*{theorem LL*}{Previous results for inhomogeneous metric Diophantine approximation}
\newtheorem*{theorem LLL*}{Previous results for multiplicative metric Diophantine approximation}

\numberwithin{equation}{section}

\renewcommand{\le}{\leq}
\renewcommand{\ge}{\geq}

\newcommand{\x}{\mathbf{x}}

\newcommand{\q}{\mathbf{q}}

\def\le{\leqslant} \def\ge{\geqslant}

\def \leq {\le}
\def \geq {\ge}

\newcommand{\Addresses}{{
		\bigskip
		\footnotesize
		
		\medskip
		
		H.~Yu, 
		\textsc{Department of Pure Mathematics and Mathematical Statistics, Centre for Mathematical Sciences, Cambridge, CB3 0WB, UK}\par\nopagebreak
		\textit{E-mail address:} \texttt{hy351@maths.cam.ac.uk}
}}

\title{On the metric theory of multiplicative Diophantine approximation}
\author{ Han Yu \footnote{Supported by the European Research Council (ERC) under the European Union’s Horizon 2020 research and innovation programme (grant agreement No. 803711), and indirectly by Corpus Christi College, Cambridge}}

\date{}

\begin{document}
	\maketitle
	\begin{abstract}
		In 1962, Gallagher proved an higher dimensional version of Khintchine's theorem on Diophantine approximation.  Gallagher's theorem states that for any non-increasing approximation function $\psi:\mathbb{N}\to (0,1/2)$ with $\sum_{q=1}^{\infty} \psi(q)\log q=\infty$ and $\gamma=\gamma'=0$ the following set
		\[
		\{(x,y)\in [0,1]^2: \|qx-\gamma\|\|qy-\gamma'\|<\psi(q) \text{ infinitely often}\}
		\]
		has full Lebesgue measure. Recently, Chow and Technau proved a fully inhomogeneous version (without restrictions on $\gamma,\gamma'$) of the above result.
		
		In this paper, we prove an Erd\H{o}s-Vaaler type result for fibred multiplicative Diophantine approximation. Along the way, via a different method,  we prove a slightly weaker version of Chow-Technau's theorem with the condition that at least one of $\gamma,\gamma'$ is not Liouville. We also extend Chow-Technau's result for fibred inhomogeneous Gallagher's theorem for Liouville fibres.
	\end{abstract}
	
	
	\section{Introduction}
	\subsection{Background}
	Motivated by the work of Beresnevich, Haynes and Velani \cite{BHV}, in this paper, we study metric multiplicative Diophantine approximation. We mainly use an elementary method introduced in \cite{Yu2}. We will focus on dimension two, but our strategy also works for higher-dimensional cases. Let $\gamma,\beta,\gamma'$ be three real numbers and $\psi:\mathbb{N}\to \mathbb{R}^+$ be a  function (approximation function). We will be interested in the following set,
	\[
	W(\psi,\beta,\gamma,\gamma')=\{x\in [0,1]: \|qx-\gamma\|\|q\beta-\gamma'\|<\psi(q)\text{ infinitely often}\},
	\]
	where $\|x\|$ for $x\in\mathbb{R}$ is the distance between $x$ and the integer set $\mathbb{Z}.$ More specifically, we want to know under which conditions (for $\psi,\beta,\gamma,\gamma'$) does $W(\psi,\beta,\gamma,\gamma')$ have positive or even full Lebesgue measure. This is a challenging problem in the study of metric Diophantine approximations. Recently, there have been many breakthroughs in this direction, and we will briefly introduce them. The study of $W(\psi,\beta,\gamma,\gamma')$ is closely related to, if not contained in, the study of inhomogeneous Diophantine approximation, where the central object to study is
	\[
	W(\psi,\gamma)=\{x\in [0,1]: \|qx-\gamma\|<\psi(q)\text{ infinitely often}\}.
	\]
	If we define
	\[
	\psi'(q)=\frac{\psi(q)}{\|q\beta-\gamma'\|},
	\]
	then $W(\psi,\beta,\gamma,\gamma')=W(\psi',\gamma).$ If we put $\gamma=0$, then the study of $W(\psi,0)$ is the classical metric Diophantine approximation. It is almost impossible to avoid mentioning the following result,
	\begin{theorem*}[Duffin-Schaeffer conjecture/Koukoulopoulos-Maynard theorem, \cite{DS}, \cite{KM2019}]
		Let $\psi$ be an approximation function with $\sum_q \psi(q)\phi(q)/q=\infty.$ then $W(\psi,0)$ has full Lebesgue measure, otherwise if $\sum_q \psi(q)\phi(q)/q<\infty$ then for Lebesgue almost all $x$, there are only finitely many coprime pairs $(p,q)$ such that $|x-p/q|<\psi(q)/q$. Here, $\phi(.)$ is the Euler phi function.
	\end{theorem*}
	
	From the above result, the Lebesgue measure of $W(\psi,\beta,0,\gamma')$ can be understood very well. However, when $\gamma\neq 0,$ much less are known.
	
	\begin{theorem LL*}\footnote{This list is not complete.}
		\begin{itemize}
			\item {Sz\"{u}sz's theorem \cite{Szusz}}: If $\psi$ is non-increasing and $\sum_{q=1}^{\infty} \psi(q)=\infty$ then $W(\psi,\gamma)$ has full Lebesgue measure for all real number $\gamma.$
			\item {Ram\'{i}rez's examples \cite{Ramirez}}: Without the monotonicity of the approximation function $\psi$, the condition $\sum_{q=1}^{\infty} \psi(q)=\infty$ alone cannot ensure $W(\psi,\gamma)$ to have positive Lebesgue measure. 
			\item {Extra divergence \cite{Yu}}: For each $\epsilon>0,$ if $\sum^{\infty}_{q=1} q^{-\epsilon} \psi(q)=\infty,$ then for all number $\gamma,$ $W(\psi, \gamma)$ has full Lebesgue measure.
			
			\item {Erd\H{o}s-Vaaler type \cite{Yu2}}: Let $\psi$ be an approximation function with $\psi(q)=O(q^{-1}(\log\log q)^{-2})$ and $\sum_{q=1}^{\infty} \psi(q)=\infty.$ Then for each non-Liouville $\gamma$ we have
			$
			|W(\psi,\gamma)|=1.$
		\end{itemize}
	\end{theorem LL*}
	It is very difficult to use the above results to study $W(\psi,\beta,\gamma,\gamma')$ via $W(\psi',\gamma).$ This is mainly because that the extra conditions (monotonicity, upper bound, extra divergence) are difficult to be tested for approximation functions of form $\psi(q)/\|q\beta-\gamma'\|.$ 
	\begin{theorem LLL*}\footnote{This list only contains results before the year 2019 and it is not complete.}
		
		\begin{itemize}
			\item {Gallagher,\cite{Gallagher62}:} Let $\psi$ be a non-increasing approximation function with $\sum_{q=1}^{\infty} \psi(q)\log q=\infty.$ For Lebesgue almost all $(x,y)\in [0,1]^2$, there are infinitely many integers $q$ with
			\[
			\|qx\|\|qy\|\leq\psi(q).
			\]
			\item {Beresnevich-Haynes-Velani, \cite[Theorem 2.3]{BHV}:} Let $\psi$ be a non-increasing approximation function with $\sum_{q=1}^{\infty} \psi(q)\log q=\infty.$ Then for all real number $\gamma,$ Lebesgue almost all $(x,y)\in [0,1]^2,$ there are infinitely many integers $q$ with
			\[
			\|qx-\gamma\|\|qy\|\leq  \psi(q).
			\]
			
			\item{Chow \cite{C18}, Koukoulopoulos-Maynard \cite{KM2019}:} Let $\psi$ be a non-increasing approximation function with $\sum_{q=1}^{\infty} \psi(q)\log q=\infty.$ Then for all non-Liouville number $\beta$ and all real number $\gamma',$ the set $W(\psi,\beta,0,\gamma')$ has full Lebesgue measure.
		\end{itemize}
	\end{theorem LLL*}
	The central problem in this area is the following conjecture by Beresnevich, Haynes and Velani \cite[Conjecture 2.1]{BHV} which is now proved by Chow and Technau in \cite{CT}.
	\begin{conjecture}\label{BHVC}
		Let $\psi$ be a non-increasing approximation function with $\sum_{q=1}^{\infty} \psi(q)\log q=\infty.$ Then for all real numbers $\gamma, \gamma'$ Lebesgue almost all $(x,y)\in [0,1]^2,$ there are infinitely many integers $q$ with
		\[
		\|qx-\gamma\|\|qy-\gamma'\|\leq  \psi(q).
		\]
	\end{conjecture}
	The proof in \cite{CT} relies on a fine analysis of the arithmetic and additive structures of Bohr sets. In this paper, we shall provide a different (and more elementary) proof for the following slightly weaker result.
	\begin{corollary}[Corollary of Theorem \ref{MAIN MONNTON}]\label{Coro2}
		Let $\psi$ be a non-increasing approximation function with $\sum_{q=1}^{\infty} \psi(q)\log q=\infty.$ Then for all non-Liouville number $\gamma$, real number $\gamma'$ and Lebesgue almost all $(x,y)\in [0,1]^2,$ there are infinitely many integers $q$ with
		\[
		\|qx-\gamma\|\|qy-\gamma'\|\leq  \psi(q).
		\]
	\end{corollary}
	The extra non-Liouville condition is not necessary because of the results in \cite{CT}. However, under this condition, one may obtain a stronger result. To be more concrete and precise, we state the following conjecture.
	\begin{conjecture}\label{con2}
		Let $\psi$ be an approximation function with $\sum_{q=1}^{\infty} \psi(q)\log q=\infty$ and $\psi(q)=O(q^{-1}(\log q)^{-2}).$ Then for all non-Liouville number $\gamma$ and real number $\gamma'$ Lebesgue almost all $(x,y)\in [0,1]^2,$ there are infinitely many integers $q$ with
		\[
		\|qx-\gamma\|\|qy-\gamma'\|\leq  \psi(q).
		\]
	\end{conjecture}
	We will prove (in Theorem \ref{MAIN}) a fibred version of Conjecture \ref{con2} under a slightly stronger requirement on $\psi,$ i.e. $\psi(q)=O(q^{-1}(\log q)^{-2}(\log\log q)^{-2}).$ Unlike the situation with monotonic $\psi,$ the unfibred statement is no longer a direct consequence of the fibred version and Conjecture \ref{con2} is still open even with this stronger requirement on $\psi.$
	
	It is interesting to continue this direction and try to obtain a sharp result for multiplicative Diophantine approximations. For example, what would be the Duffin-Schaffer conjecture in the multiplicative setting? 
	\subsection{Results in this paper}
	Before we state the main results, we mention that for monotonic approximation functions, the optimal results are proved in \cite{CT}. Our results are weaker in the sense that we need to require stronger Diophantine conditions for some parameters. For example, Corollary \ref{Coro2} requires one of the shift parameters $\gamma,\gamma'$ to be not Liouville. Our method also extends to deal with non-monotonic approximation functions. In this situation, we believe that some of the stronger Diophantine conditions are in fact necessary.  
	\subsubsection{Diophantine parameters}
	Our method works very well when some Diophantine conditions are presented. Before we state the results, some standard terminologies are needed. 
	\begin{defn}
		We say that a real number $\gamma$ is Diophantine if there are positive numbers $\rho,c$  such that $\|n\gamma\|\geq c n^{-\rho}$ for all $n\geq 1.$ The infimum of all such possible $\rho$ (we allow $c$ to change) is referred to as the Diophantine exponent of $\gamma.$ More generally, let $\gamma_1,\dots,\gamma_r$ be $r\geq 2$ real numbers which are $\mathbb{Q}$-linearly independent. We say that $(\gamma_1,\dots,\gamma_r)$ is a Diophantine $r$-tuple if there are positive numbers $\rho,c$ such that
		\[
		\|n_1\gamma_1+\dots+n_r\gamma_r\|\geq c(\max\{|n_1|,\dots,|n_r|\})^{-\rho}
		\]
		for all $n_1,\dots,n_r\in \mathbb{Z}.$ The infimum of all such possible $\rho$ is referred to as the Diophantine exponent\footnote{In some literatures, this is called the dual exponent.} of $(\gamma_1,\dots,\gamma_r).$
	\end{defn}
	In general, the Diophantine exponent of $r$ $\mathbb{Q}$-linearly independent numbers is at least $r.$ It is possible to see that if $(\gamma_1,\dots,\gamma_r)$ is Diophantine then $(M_1\gamma_1,\dots,M_r\gamma_r)$ is Diophantine for all non-zero integers $M_1,\dots,M_r.$ Further more, their Diophantine exponents are equal although the associated constants might be different. Examples of Diophantine tuples include tuples of $\mathbb{Q}$-linearly independent algebraic numbers, for example, $(\sqrt{2},\sqrt{3}).$ Other examples includes tuples of natural logarithms of algebraic numbers, for example, $(\ln 2,\ln 3).$
	
	We first provide the following Erd\H{o}s-Vaaler type result.
	\begin{theorem}\label{MAIN}[main Theorem I]
		Let $\psi(q)=O((q\log q(\log\log q)^2)^{-1})$ be an approximation function. Let $\gamma,\beta$ be irrational numbers such that $(\gamma,\beta)$ is Diophantine. Let $\gamma'$ be a real number.  Then for all small enough $\omega>0$ (in a manner that depends on $(\gamma,\beta),\gamma$ and $\beta$), if the following condition holds,
		\begin{align}
		\sum_{\substack{q\in \mathbb{N}\\\|q\beta-\gamma'\|\in [q^{-\omega},1)}}\frac{\psi(q)}{\|q\beta-\gamma'\|}=\infty,\label{Div}
		\end{align}
		then $|W(\psi,\beta,\gamma,\gamma')|=1.$
	\end{theorem}
	\begin{remark}
		The power $2$ in $q\log q (\log\log q)^2$ is by no means the optimal value with our method. It is very likely to be improved to $1$ or even $0.$  The value $2$ comes from a rather crude bound (\ref{II'1}) in the proof of this theorem. See Section \ref{proofofmain} for more details.
	\end{remark}
	\begin{remark}\label{Conv}
		The divergence condition (\ref{Div}) in the above results looks rather technical. A way of viewing it is that after fixing $\gamma,\beta$ and $\gamma'$ we can choose a small enough $\omega>0$ and consider the set of integers $q$ with $\|q\beta-\gamma'\|\in [q^{-\omega},1].$ Restricted to this set of integers, we can design approximation functions $\psi$ freely subject to a upper bound condition and a divergence condition (\ref{Div}). In fact, under the condition that $\psi$ is supported on where $\|q\beta-\gamma'\|\in [q^{-\omega},1]$ the condition (\ref{Div}) is also necessary.
	\end{remark}
	If $\psi$ is monotonic, then we can prove the following fibred Chow-Technau's theorem with Diophantine parameters.
	\begin{theorem}\label{MAIN MONNTON}[Chow-Technau's theorem for Diophantine parameters]
		Let $\psi(q)$ be a non-increasing approximation function. Let $\gamma,\beta$ be irrational numbers such that $(\gamma,\beta)$ is Diophantine. Let $\gamma'$ be a real number. Suppose that $\sum_{q=1}^{\infty} \psi(q)\log q=\infty,$ then
		\[
		|W(\psi,\beta,\gamma,\gamma')|=1.
		\]
	\end{theorem}
	\begin{remark}
		This is a weaker version of a result recently proved by Chow and Technau. In fact, Chow-Technau's theorem(\cite[Corollary 1.10]{CT} with $k=2$) says that the conclusion holds by only assuming  that $\beta,\gamma$ are irrational and $\beta$ is non-Liouville. Theorem \ref{MAIN MONNTON} holds, for example, when $(\gamma,\beta)=(\sqrt{2},\sqrt{3})$ or $(\sqrt{3},\log 2).$ However, Chow-Technau's result can be used, for example, when $(\gamma,\beta)=(\pi,e)$ in which case one cannot use Theorem \ref{MAIN MONNTON} directly as it is not known whether $(\pi,e)$ is Diophantine or not.
	\end{remark}
	\subsubsection{General parameters}
	Our method can be also used to consider the case when $(\gamma,\beta)$ is not assumed to be Diophantine.  Let $\sigma(.)$ be a function taking integer variables and positive values so that $\sigma(N)$ is the best Diophantine exponent for $(\gamma,\beta)$ up to height $N.$ That is to say, $\sigma(N)$ is the infimum of all numbers $\sigma>0$ such that for all $-N\leq k_1,k_2\leq N$ with $k_1,k_2$ not both zeros,
	\[
	\|k_1\gamma+k_2\beta\|\geq \max\{k_1,k_2\}^{-\sigma}
	\] 
	In particular, if $(\gamma,\beta)$ is Diophantine, then $\sigma(.)$ is a bounded function. In general, $\sigma(.)$ is a non-decreasing function. When there are possible confusions, we write $\sigma_{(\gamma,\beta)}$ to indicate the $\sigma$ function associated with the pair $(\gamma,\beta).$  Similarly, one can define the $\sigma$ function for real numbers. For example $\sigma_\gamma(.)$ would be just be defined as $\sigma_{(\gamma,\beta)}$ but with $k_2=0$ throughout. Thus if $\sigma_\gamma$ is unbounded then $\gamma$ is Liouville. We do not exclude rational numbers.  If $\gamma$ is rational, then it is not Diophantine in this sense. In this case, $\sigma_\gamma$ is not a well defined function as it attains $\infty.$

	\begin{theorem}\label{MAIN2}[main Theorem II]
		Let $\psi(q)=O((q\log q (\log\log q))^{-3})$ be an approximation function. Let $\gamma,\beta$ be irrational numbers such that $\sigma_{(\gamma,\beta)}(q)=O ((\log\log\log q)^{1/2}).$ Let $\epsilon>0.$ Then there is a small number $c>0$ such that the following holds.
		
		Let $\omega:\mathbb{N}\to \mathbb{R}$ be the function
		\[
		\omega(q)=
		\begin{cases}
		1 & \log\log q\leq 1\\
		c/(\log\log\log q)^{1/2} & \log\log q>1
		\end{cases}.
		\]
		Let $\gamma'$ be a real number. If
		\begin{align*}
		\sum_{\substack{q\in\mathbb{N}\\\|q\beta-\gamma'\|\in [q^{-\omega(q)},1)}}\frac{\psi(q)}{\|q\beta-\gamma'\|}=\infty,
		\end{align*}
		then $|W(\psi,\beta,\gamma,\gamma')|=1.$
	\end{theorem}
	Again, if $\psi$ is monotonic, then it is possible to prove  a result with cleaner conditions.
	\begin{theorem}\label{MAIN MONO2}
		Suppose that $\psi$ is a monotonic approximation function such that
		\[
		\sum_{q=1}^{\infty} \psi(q)\frac{\log q}{(\log\log q)^{1/2}}=\infty.
		\]
		Let $\gamma,\beta$ be irrational numbers such that $\sigma_{(\gamma,\beta)}(q)=O (\log\log q)^{1/2}).$ Let $\gamma'$ be a real number. Then $|W(\psi,\beta,\gamma,\gamma')|=1.$
	\end{theorem}
	For example the approximation function $\psi(q)=1/(q(\log q)^2 (\log\log q)^{1/2})$ satisfies the above condition. We remark that it is not possible to completely drop the Diophantine condition for $(\gamma,\beta).$ In fact, \cite[Theorem 1.14]{CT} shows that under the hypothesis (for $\psi$) in Theorem \ref{MAIN MONO2}, there exist infinitely many choices of $\gamma,\beta,\gamma'$ such that $|W(\psi,\beta,\gamma,\gamma')|=0.$  We believe that the joint Diophantine condition for $(\gamma,\beta)$ can be reduced to a Diophantine condition for $\beta$ only.
	\begin{conjecture}
		Suppose that $\psi$ is a monotonic approximation function such that
		\[
		\sum_{q=1}^{\infty} \psi(q)\frac{\log q}{(\log\log q)^{1/2}}=\infty.
		\]
		Let $\beta$ be an irrational number such that $\sigma_{\beta}(q)=O ((\log\log q)^{1/2}).$ Let $\gamma,\gamma'$ be real numbers. Then $|W(\psi,\beta,\gamma,\gamma')|=1.$
	\end{conjecture}

	\section{Notation}\label{notation}
	
	\begin{itemize}
		\item $d(.)$ is the divisor function. We will not use any other 'reserved' arithmetic functions. For example,  $\omega$ in this paper is NOT the distinct prime factors function.
		\item $A^{\psi,\gamma}_q$: 
		Let $\psi$ be an approximation function and $\gamma$ be a real number. For each integer $q\geq 1,$ we use $A^\gamma_q$ to denote the set
		\[
		A^{\psi,\gamma}_q=\{x\in [0,1]: \|qx-\gamma\|<\psi(q)\}.
		\]
		We can assume that $\psi(q)<1/2$ for all $q\geq 1.$ In fact, if there are infinitely many $q$ with $\psi(q)\geq 1/2,$ then $W(\psi,\gamma)$ would be the whole unit interval. If $\gamma,\psi$ is clear from the context, we will write $A_q$ instead of $A^{\psi,\gamma}_q.$
		\item $\chi_A$: The indicator function of a set $A.$
		\item $B(x,r)$: Metric ball centred at $x$ with radius $r,$ where $r>0$ and $x$ belongs to a metric space.
		\item $I_r=\chi_{B(0,r)}.$ Namely, $I_r:\mathbb{R}\to\{0,1\}$ is such that $I_r(x)=1$ if and only if $x\in [-r.r].$
		\item $\Delta_{\psi}(q,q')$: The value $q\psi(q')+q'\psi(q)$, where $\psi$ is a given approximation function and $q,q'$ are positive integers. When $\psi$ is clear from the context, we write it as $\Delta(q,q').$
		\item $\|x\|$: The distance of a real number $x$ to the set of integers.
		\item $\{x\}$: The unique number in $(-1/2,1/2]$ with $\{x\}-x$ being an integer.
		\item $\log $: Base $2$ logarithmic function.
		\item $\mathcal{I}_M$: The collection of intervals $I\subset [0,1]$ of length $1/M$ and with endpoints in $M^{-1}\mathbb{N},$ where $M\geq 1$ is an integer.
		\item $|A|$: The Lebesgue measure of $A\subset \mathbb{R}$ where $A$ is a Lebesgue measurable set.
		\item Conditioned sums: Let $A$ be a countable set of positive numbers. It is standard to use $\sum_{x\in A} a$ for the sum of elements in $A$ since the order of summation does not matter. Sometimes, $A$ can be described by some properties, say $P.$ By saying that under the condition $P$
		\[
		\sum_{x\in (0,\infty)}x,
		\]
		we actually mean $\sum_{x\in A} x.$ This has an advantage when $P$ is a complicated property, and it would be too long to appear under the summations symbol. More generally, let $B$ be a set of positive numbers, by saying that under the condition $P$
		\[
		\sum_{x\in B}x,
		\]
		we mean
		\[
		\sum_{\substack{x\in B\\x\text{ saisfies } P}}x.
		\]
		\item Asymptotic symbols: For two functions $f,g:\mathbb{N}\to (0,\infty)$ we use $f=O(g)$ to mean that there is a constant $C>0$ with
		\[
		f(q)\leq Cg(q)
		\]
		for all $q\geq 1.$
		We use $f=o(g)$ to mean that 
		\[
		\lim_{q\to\infty} \frac{f(q)}{g(q)}=0.
		\]
		For convenience, we also use $O(g), o(g)$ to denote an auxiliary function $f$ with the property that $f=O(g)$, $o(g)$ respectively. The precise form of the function $f$ changes across the contexts, and it can always be explicitly written down.
	\end{itemize}

	\section{Preliminaries}
	There are several standard results that will be needed in the proofs of the main results. The first one is the Borel-Cantelli lemma. The following version can be found in \cite[Proposition 2]{BDV ref}.
	\begin{lemma}[Borel-Cantelli]\label{Borel}
		Let $(\Omega, \mathcal{A}, m)$ be a probability space and let $E_1, E_2, \ldots \in \mathcal{A}$ be a sequence of events in $\Omega$ such that $\sum_{n=1}^{\infty}{m(E_n)} = \infty$. Then 
		\[m(\limsup_{n \to \infty}{E_n}) \geq \limsup_{Q \to \infty}{\frac{\left(\sum_{s=1}^{Q}{m(E_s)}\right)^2}{\sum_{s,t=1}^{Q}{m(E_s \cap E_t)}}}.\]
		If $\sum_{n=1}^{\infty}{m(E_n)} < \infty,$ then $m(\limsup_{n \to \infty}{E_n})=0.$
	\end{lemma}
	We also need the following discrepancy result, \cite[Section 1.4.2]{DT97}. Let $k\geq 1$ be an integer. Let $(a_1,\dots,a_k)\in \mathbb{T}^k$ be a $\mathbb{Q}$-linear independent vector. Let $I\subset\mathbb{T}^k$ be a box (a Cartesian product of intervals). We wish to count
	\[
	S_I(Q,a_1,\dots,a_k)=\#\{q\in \{1,\dots,Q\}:  q(a_1,\dots,a_k)\in I\}.
	\]
	We write
	\[
	S_I(Q,a_1,\dots,a_k)=Q |I|+ E^I_{a_1,\dots,a_k}(Q).
	\]
	We have the following upper bound for the error term $E^I_{a_1,\dots,a_k}(Q)$
	\[
	E^I_{a_1,\dots,a_k}(Q)=o_I(Q).
	\]
	This follows from the ergodicity of irrational rotation $+(a_1,\dots,a_k)$ on $\mathbb{T}^k.$ In some special case, we have a much better understanding of the error term. Suppose that $(a_1,\dots,a_k)$ is Diophantine, for example, when $a_1,\dots,a_k$ are $\mathbb{Q}$-linearly independent algebraic numbers. Then there are a number $\rho\in (0,1)$, a constant $C>0,$ such that for all boxes $I,$ all  integers $Q\geq 1,$  
	\[
	|E^I_{a_1,\dots,a_k}(Q)|\leq CQ^{\rho}. 
	\]
	The infimum of all such $\rho$ is called the discrepancy exponent of $(a_1,\dots,a_k).$ We have the following result \cite[Theorem 1.80]{DT97} which related discrepancy exponent to Diophantine exponent.
	\begin{lemma}\label{Discrepancy bound}
		Let $a_1,\dots,a_k$ be $k\geq 1$ $\mathbb{Q}$-linearly independent numbers. Suppose that the Diophantine exponent of $(a_1,\dots,a_k)$ is $\rho$ (which is at least $k$). Then the discrepancy exponent of $(a_1,\dots,a_k)$ is at most $1-1/\rho.$
	\end{lemma}
	We use $D_{a_1,\dots,a_k}(Q)$ to denote
	\[
	\sup_{I}E^I_{a_1,\dots,a_k}(Q)
	\]
	where the $\sup$ is taken over all boxes $I\subset [0,1]^k.$ Let $M_1,\dots,M_k$ be integers. It can be checked that
	\[
	D_{M_1a_1,\dots,M_ka_k}(Q)\leq M_1\dots M_k D_{a_1,\dots,a_k}(Q).
	\]

	Let $\psi$ be an approximation function and $\gamma$ be a real number. In order to use Lemma \ref{Borel}, we need to estimate the size of intersections $A_q\cap A_{q'}.$ We have the following result from \cite[Lemma 4.1]{Yu2}. 
	\begin{lemma}\label{master}
		Let $H>2$ be an integer. Let $\psi$ be an approximation function and $\gamma$ be an irrational number. For integers $1\leq q'<q$ such that $\Delta(q',q)<H\gcd(q,q')$ we have the following estimate
		\[
		|A_q\cap A_{q'}|\leq 2(2H+1)\min\{\psi(q)/q, \psi(q')/q'\} \gcd(q,q') I_{\Delta(q,q')/\gcd(q,q')}(\{\gamma(q'-q)/\gcd(q,q')\}).
		\]
		Otherwise when $\Delta(q,q')\geq H\gcd(q,q')$, we have
		\[
		|A_q\cap A_{q'}|\leq 4(1+C_0/(2H))\psi(q)\psi(q'),
		\]
		where $C_0>1$ is an absolute constant.
	\end{lemma}
	\section{Outline of the proofs}
	The proofs of the main theorems are very similar. Here we provide an outline of the proof of Theorem \ref{MAIN}. Let $\beta,\gamma,\gamma'$ be real numbers and $\psi$ be an approximation function.  We want to consider the limsup of the sets $\{A_q=A^{\psi',\gamma}_q\}_{q\geq 1}$ where 
\[
\psi'(q)=\frac{\psi(q)}{\|q\beta-\gamma'\|}.
\]
We will use the divergence Borel-Cantelli lemma (Lemma \ref{Borel}). For this, we need to study the intersections $A_q\cap A_{q'}$ for different positive integers $q,q'.$ In this direction, Lemma \ref{master} offers us some helpful information. We need to use this lemma to provide some estimates on the sum
\[
\sum_{1\leq q'\leq q}|A_{q'}\cap A_q|
\]
which will be helpful for the divergence Borel-Cantelli lemma. To estimate the above sum, we see that it is helpful to study the sequence of $0,1$'s,
\[
\{I_{\Delta(q,q')/\gcd(q,q')}(\{\gamma(q'-q)/\gcd(q,q')\})\}_{1\leq q'\leq q}.
\]
This sequence is essentially driven by the irrational rotation (on the unit interval) with rotating angle $\gamma.$ Due to the uniform ergodicity of irrational rotations, we expect that the above sequence over $\{0,1\}$ is `well distributed'. More precisely, we expect that the appearances of $1$'s are almost periodic with gaps around $1/\{\gamma\}.$ Up to some quantifiable measures (depending on $\gamma$), we can treat this sequence as the constant sequence of $1$'s. With this strategy, we have resolved the inhomogeneous parameter $\gamma.$ 

Next, we consider the inhomogeneous parameter $\gamma'.$ It is now carried by the modified approximation function $\psi'.$ The factor $\|q\beta-\gamma'\|$ makes $\psi'$ appear somehow irregular. Luckily, $\psi'$ is not too irregular. The strategy is to flatten the range of $\psi'.$ For example, let $u\in (0,1/2)$ be a number and we consider set $G_u$ of $q$'s with
\[
\|q\beta-\gamma'\|\in [u,2u].
\]
On $G_u$, we can essentially ignore the effect of the factor   $\|q\beta-\gamma'\|$. Here, the set $G_u$ is controlled by the irrational rotation with angle $\beta.$ Again, by ergodicity, we expect that $G_u$ is `well distributed' in $\mathbb{N}.$  We can then use this fact and Lemma \ref{master} to obtain good estimates for
\[
\sum_{\substack{1\leq q'\leq q\\ \|q'\beta-\gamma'\|\in [u,2u]}}|A_{q'}\cap A_q|
\]
with various $u.$ Then we can sum the above estimates for different values of $u$ and obtain a good estimate for 
\[\sum_{1\leq q'\leq q}|A_{q'}\cap A_q|.\]
This is the way we treat the inhomogeneous parameter $\gamma'.$

From here, we see that two irrational rotations (with parameters $\gamma$ and $\beta$) dominate the above arguments. A very crucial point here is that we need to treat them at the same time. This effectively leads to the consideration of the two-dimensional irrational rotation with angle $(\gamma,\beta).$ In order for this to work, we need that the pair $(\gamma,\beta)$ is Diophantine. This condition basically says that the orbit $\{(q\gamma,q\beta)\mod \mathbb{Z}^2\}_{q\geq 1}$ is 'well distributed' over $[0,1]^2.$ With this regularity, we can partially ignore the effects caused by the inhomogeneous parameters $\gamma,\gamma'$ and the result will follow.
	\section{Bounding intersections}
	The main purpose of this section is to prove several lemmas on summing over measures of intersections. Throughout the rest of this paper, we let $H>2$ be an integer and $C_0$ be as in Lemma \ref{master}. The results of this section are the following  Lemmas \ref{Lemma} and \ref{Lemma3}. The proofs are complicated and the reader can skip the proofs in this section  for now and read Section \ref{proofofmain} to see how to used them to prove the main theorems.
	
	In what follows, for each integer $q>1,$ define
	\[
	F(q)=\sum_{r|q} \frac{\log r}{r}.
	\]
	Suppose that $\psi$ is an approximation function. Let $\beta,\gamma,\gamma'$ be real numbers and $\omega$ be a positive number. Define
	\begin{align}
	\psi_{\beta,\gamma',\omega}^{'}(q)=
	\begin{cases}
	\frac{\psi(q)}{\|\beta q-\gamma'\|} & \|q\beta-\gamma'\|\in [q^{-\omega},1]\\
	0 & \text{else}
	\end{cases}.\label{psi}
	\end{align}
	Often, $\beta,\gamma'$ are clear from the context. If so, we write $\psi'_{\omega}$ instead of $\psi'_{\beta,\gamma',\omega}.$ In addition, if $\omega$ is also clear from the context, we simply write $\psi'.$ Here, $\omega$ may not be a constant as $q$ varies. In fact, we will consider the situation when $\omega:\mathbb{N}\to [0,\infty)$ is a function. In this case, we write $\psi'_{\omega(.)}$ for the function
	\[
	q\to \psi'_{\omega(q)}(q).
	\]
	
	\begin{lemma}\label{Lemma}
		Suppose that $\psi(q)=O(q^{-1}(\log q)^{-1}(\log\log q)^{-2}).$ Let $\gamma$ be an irrational non-Liouville number, $\beta$ be a real number such that $(\beta,\gamma)$ is Diophantine and $\gamma'$ be a real number. Then there is a positive number $\omega_0$ depending on the Diophantine exponents of $\gamma,\beta, (\beta,\gamma)$ such that for all $\omega\in (0,\omega_0),$ the approximation function $\psi'=\psi'_{\beta,\gamma',\omega}$ satisfies
		\[
		\sum_{1\leq q'<q} |A^{\psi',\gamma}_q\cap A^{\psi',\gamma}_{q'}|=O\left(\psi'(q)+\frac{\psi'(q)}{(\log\log q)^2}F(q)\right)+4(1+C_0/(2H))\psi'(q)\sum_{1\leq q'< q} \psi'(q')
		\]
		where the implicit constant can be made explicit.
	\end{lemma}
	\begin{remark}\label{Remark}
		If we only assume $\psi(q)=O(q^{-1}(\log q)^{-1}),$ then we have the following conclusion
		\[
		\sum_{1\leq q'<q} |A^{\psi',\gamma}_q\cap A^{\psi',\gamma}_{q'}|=O\left(\psi'(q)+\psi'(q)\sum_{r|q} \frac{\log r}{r}\right)+4(1+C_0/(2H))\psi'(q)\sum_{1\leq q'< q} \psi'(q').
		\]
	\end{remark}
	
	\begin{lemma}\label{Lemma3}
		Suppose that $\psi(q)=O(q^{-1}(\log q)^{-1}(\log\log q)^{-3}).$ Let $(\gamma,\beta)$ be a pair of numbers whose $\sigma$ function satisfies $\sigma(q)=O((\log\log q)^{1/2}).$ Then there is a constant $C_1>0$ such that  the approximation function $\psi'=\psi'_{\beta,\gamma',\omega(.)}$ satisfies
		\[
		\sum_{1\leq q'<q} |A^{\psi',\gamma}_q\cap A^{\psi',\gamma}_{q'}|=O\left(\psi'(q)+\frac{\omega(q)\psi'(q)}{(\log\log q)^{3}}F(q)\right)+4(1+C_0/(2H))\psi'(q)\sum_{1\leq q'< q} \psi'(q'),
		\]
		where $\omega(q)=C_1/\sigma(q).$  If $\psi(q)=O(q^{-1}(\log q)^{-1} (\log\log q)^{1/2}),$ then as in above (with $\omega(q)=C_1/(\log\log q)^{1/2}$) we have
		\[
		\sum_{1\leq q'<q} |A^{\psi',\gamma}_q\cap A^{\psi',\gamma}_{q'}|=O\left(\psi'(q)+\psi'(q)F(q)\right)+4(1+C_0/(2H))\psi'(q)\sum_{1\leq q'< q} \psi'(q').
		\]
	\end{lemma}
	\begin{remark}
		The second part of this lemma will be very useful in  the case when we know that $\psi'$ is supported on where $F(q)=\sum_{r|q} \log r/r$ is uniformly bounded from above. 
	\end{remark}
	\subsection{A counting lemma with Diophantine parameters: proof of Lemma \ref{Lemma}}
	First, we will deal with the case when $(\gamma,\beta)$ is a Diophantine pair. All the main ideas will be included in this special case. Later on, we will add extra technical arguments to weaken this Diophantine condition. \footnote{Of course, the integrity of this manuscript suffers from this non-linear proving style. The reader may need to constantly re-read this section in order to be able to finish the later proofs. We have carefully set up the breakpoints inside the proof such that the reader can conveniently come back and re-read what will be needed.}In this way, those extra arguments will not block our view too much, at least for now. Weakening the Diophantine condition is not for free. We shall see later that the less Diophantine condition we ask for $(\gamma,\beta),$ the less information we can use from $\psi.$ 
	
	\begin{proof}[Proof of Lemma \ref{Lemma}]
		Without loss of generality we assume that $q\geq 1024.$ We also assume that $\psi(q)\leq q^{-1}(\log q)^{-1}(\log\log q)^{-1}.$ Let $\omega>0$ be a number that will be determined later. In the proof, we will need to introduce a constant $C>0$ whose value will be updated throughout the proof. Consider $\psi'=\psi_{\beta,\gamma'\omega}.$ In what follows, we will use $\psi'$ as the default approximation function and we write $A$ for $A^{\psi',\gamma}$ and $\Delta$ for $\Delta_{\psi'}.$ By Lemma \ref{master}, for each integer $H>2$, if $\Delta(q,q')<H\gcd(q,q')$ then we have
		\[
		|A_q\cap A_{q'}|\leq 2(2H+1)\min\{\psi'(q)/q,\psi'(q')/q'\}\gcd(q,q')I_{\Delta(q,q')/\gcd(q,q')}(\{\gamma(q'-q)/\gcd(q,q')\})
		\]
		otherwise, for a $C_0>0,$
		\[
		|A_q\cap A_{q'}|\leq 4(1+C_0/(2H))\psi'(q)\psi'(q').
		\]
		Now we wish to estimate the sum
		\begin{align}
		\sum_{\substack{1\leq q'<q\\\|q'\beta-\gamma'\|\in [{q'}^{-\omega},1)}} \min\left\{\frac{\psi'(q')}{q'},\frac{\psi'(q)}{q}\right\}\gcd(q',q)I_{\Delta(q,q')/\gcd(q,q')}(\{\gamma(q'-q)/\gcd(q,q')\})\label{SUM}
		\end{align}
		There are now two different 'rotations' hidden in the above sum. The first one is inside the term
		\[
		\{\gamma (q'-q)/\gcd(q,q')\}
		\]
		and the second one is inside the value of $\psi'$ which depends heavily on $\|q\beta-\gamma'\|.$ We will use both the 'rotations' in the following argument.
		
		Before the main part of the proof, let us observe some simple facts. First, we see that for each $1>\rho_1>0,$
		\begin{align}
		\sum_{1\leq q'<q^{1-\rho_1}} \frac{\psi'(q)}{q}\gcd(q',q)I_{\Delta(q,q')/\gcd(q,q')}(\{\gamma (q'-q)/\gcd(q,q')\})\leq \psi'(q)\frac{q^{1-\rho_1}}{q}d(q)=\psi'(q)o(1).\label{1}
		\end{align}
		Next observe that
		\begin{align}
		&\sum_{1\leq q'<q} \frac{\psi'(q)}{q}\gcd(q',q)I_{\Delta(q,q')/\gcd(q,q')}(\{\gamma (q'-q)/\gcd(q,q')\})\\
		=&\frac{\psi'({q})}{q}\sum_{r|q} r \sum_{q': \gcd(q',q)=r}I_{\Delta(q,q')/r}(\{\gamma (q'-q)/r\}).\label{DIVISOR SUM}
		\end{align}
		Then we see that for each $\rho_2>1,$
		\begin{align}
		\sum_{r|q}r\sum_{\substack{q':\gcd(q',q)=r\\q'\leq q/r^{\rho_2}}}I_{\Delta(q,q')/r}(\{\gamma (q'-q)/r\})\leq \sum_{r|q}r \frac{q}{r^{1+\rho_2}}\leq \zeta(\rho_2)q.\label{2}
		\end{align}
		Therefore, to estimate (\ref{SUM}), it is enough to reduce the range of the sum to $q'\geq q^{1-\rho_1}$ as well as $q'\geq q/r^{\rho_2}$ for each divisor $r|q$ in (\ref{DIVISOR SUM})\footnote{For convenience, we record here the condition that $q'\geq \max\{q^{1-\rho_1},q/r^{\rho_2}\}$ for fixed values $1>\rho_1>0,\rho_2>1.$}. In what follows, those conditions are always implied. Those conditions imply in particularly that for large enough $q>0,$
		\[
		(\log\log q')^2\geq \frac{1}{2}(\log\log q)^2
		\]
		and
		\[
		\log q'\geq (1-\rho_1)\log q.
		\]
		So far, the above arguments do not rely on any Diophantine conditions for $(\gamma,\beta).$ We now split the expression (\ref{SUM}) according to whether $\psi'(q')q\geq \psi'(q)q'$ or $\psi'(q')q\leq \psi'(q)q'$. We consider them separately with similar arguments. There, we will see how the Diophantine condition for $(\gamma,\beta)$ is used.
		\subsubsection*{Case I: $\psi'(q')q\geq \psi'(q)q'$}
		We first deal with summand for which $\psi'(q')q\geq \psi'(q)q'.$ This implies that 
		\[
		\min\left\{\frac{\psi'(q')}{q'},\frac{\psi'(q)}{q}\right\}=\frac{\psi'(q)}{q}.
		\]
		In what follows, we will not explicitly write down this condition. For concreteness, we choose $\rho_1=1/2$ and $\rho_2=2$ in this case. Let $l\geq 0$ be an integer and we consider the set
		\[
		G_\omega^l=\{q'\geq 1, q'\in\mathbb{N}:\|q'\beta-\gamma'\|\in [2^l/{q'}^{\omega},2^{l+1}/{q'}^{\omega}]\}.
		\]
		Since $\omega$ is fixed throughout the proof, we simply write $G^l=G^l_\omega.$
		
		Now observe that (decompose $\mathbb{N}$ into sets $G^l$ with various $l$)
		\begin{align*}
		\sum_{\substack{1\leq q'<q\\\|q'\beta-\gamma'\|\in [{q'}^{-\omega},1)}} \frac{\psi'(q)}{q}\gcd(q',q)I_{\Delta(q,q')/\gcd(q,q')}(\{\gamma (q'-q)/\gcd(q,q')\})\\
		\leq \frac{\psi'(q)}{q}\sum^{l\leq \omega \log q}_{l= 0}\sum_{r|q}r \sum_{\substack{q': \gcd(q',q)=r\\q'\in G^l}}I_{\Delta(q',q)/r}(\gamma(q'-q)/r).
		\end{align*}
		For each integer $k\geq 0$ we introduce the following set
		\[
		D_k(q)=\{q'\geq 1, q'\in\mathbb{N}: q'\in [q/2^{k+1},q/2^k]\}.
		\]
		Recall that we implicitly have  the condition $q'\geq q^{1/2}$ and this will make $2^k$ to be at most $\sqrt{q}.$\footnote{We have $2^k\leq q^{1/2}.$ We also have $q/2^k\geq q/r^2,$ which forces $2^k\leq r^2.$ In general, if we do not specify the values of $\rho_1,\rho_2,$ we have $2^k\leq \min\{q^{\rho_1}, r^{\rho_2}\}.$}  Then we see that
		\begin{align*}
		\sum_{\substack{1\leq q'<q\\\|q'\beta-\gamma'\|\in [{q'}^{-\omega},1)}} \frac{\psi'(q)}{q}\gcd(q',q)I_{\Delta(q,q')/\gcd(q,q')}(\{\gamma (q'-q)/\gcd(q,q')\})\\
		\leq  \frac{\psi'(q)}{q}\sum^{l\leq \omega \log q}_{l\geq 0}\sum_{r|q}r \sum_{k:1\leq 2^k\leq q^{1/2}}\sum_{\substack{q':\gcd(q',q)=r\\q'\in G^l\\q'\in D_k(q)}}I_{\Delta(q',q)/r}(\gamma(q'-q)/r).
		\end{align*}
		If $q'\in G^l\cap D_k(q)$, then we have (recall that $(\log\log q')^2\geq (\log\log q)^2/2, \log q'\geq \log q/2$, $\psi'(q')q\geq \psi'(q)q'$ and $\psi(q)\leq q^{-1}(\log q)^{-1}(\log\log q)^{-2}$)
		\[
		\Delta(q',q)\leq 2\psi'(q')q\leq 2\frac{2^{k+1}}{\log q'(\log\log q')^2} \frac{q^{\omega}}{2^l 2^{\omega k}}\leq 8\frac{2^{k+1}}{\log q(\log\log q)^2} \frac{q^{\omega}}{2^l2^{\omega k}}.
		\]
		For each integer pair $k,l,$ we wish to estimate the sum
		\begin{align}
		&S_{k,l,r}(q) =\sum_{\substack{q': \gcd(q',q)=r\\q'\in G^l\\q'\in D_k(q)}}I_{\Delta(q',q)/r}(\gamma(q'-q)/r)&\nonumber\\
		&\leq \#\{1\leq s\leq q/r-1: sr\in G^l\cap D_k(q), \{\gamma(s-q/r)\}\in B(0,2^{k+4-l-\omega k}q^{\omega}/r\log q(\log\log q)^2)\}&\nonumber\\
		&\leq \#\{1\leq s\leq q/r-1: sr\in D_k(q), \|sr\beta-\gamma'\|\in [2^{l+\omega k}/q^{\omega},2^{l+1+\omega(k+1)}/q^{\omega}],\nonumber&\\
		&\{\gamma(s-q/r)\}\in B(0,2^{k+4-l-\omega k}q^{\omega}/r\log q(\log\log q)^2)\}.&\label{R} \end{align}
		
		If $q/(2^kr)<1$ then no non-zero multiple of $r$ is inside $D_k(q).$ If $q/(2^kr)\geq 1,$ then there are at most $3q/(2^kr)$ many possible multiples of $r$ in $D_k(q).$ We wish to use Lemma \ref{Discrepancy bound} to estimate $S_{k,l,r}(q).$ From our assumption, $(\gamma,\beta)$ is a Diophantine pair. Thus, there are numbers $\epsilon, C>0,$ such that
		\begin{align}
		S_{k,l,r}(q)\leq 8\frac{2^{k+1}}{\log q(\log\log q)^2} \frac{q^{\omega}}{2^l2^{\omega k}}\times\frac{1}{r} \times \frac{2^{l+1}2^{(k+1)\omega}}{q^{\omega}}\frac{3q}{2^k r}+C r \left(\frac{3q}{2^k r}\right)^{\epsilon}.\label{E1}
		\end{align}
		The factor $r$ in the above discrepancy error term comes from the fact that we are analysing irrational rotation with $(r\beta,\gamma)$ rather than $(\beta,\gamma).$ Now, if we ignore the condition that $q'\in G^l$ in the sum, then we have the upper bound
		\[
		S_{k,l,r}(q)\leq \#\{1\leq s\leq q/r-1: sr\in D_k(q), \{\gamma(s-q/r)\}\in B(0,2^{k+4-l-\omega k}q^{\omega}/r\log q(\log\log q)^2)\}.
		\]
		Again, Lemma \ref{Discrepancy bound} for irrational rotation with parameter $\gamma$ gives us a positive number $\epsilon'<1$ with
		\begin{align}
		S_{k,l,r}(q)\leq 8\frac{2^{k+1}}{\log q(\log\log q)^2} \frac{q^{\omega}}{2^l 2^{\omega k}}\frac{1}{r} \frac{3q}{2^k r}+C\left(\frac{3q}{2^k r}\right)^{\epsilon'}.\label{E2}
		\end{align}
		Finally, as $\gamma$ is irrational and not Liouville, there is a number $\alpha>1$ such that $S_{k,l,r}(q)=0$ whenever
		\begin{align}
		8\frac{2^{k+1}}{r\log q(\log\log q)^2} \frac{q^{\omega}}{2^l2^{\omega k}}\leq \left(\frac{q}{r}\right)^{-\alpha}. \label{C2}
		\end{align}
		To see this, observe that as $s$ ranging over $\{1,\dots,q/r-1\}$, the value of $\{\gamma(s-q/r)\}$ ranges over $\{\gamma\},\{2\gamma\},\dots,\{\gamma(q/r-1)\}.$ As $\gamma$ is irrational and not Liouville, we see that $\|n\gamma\|\geq n^{-\alpha}$ for a number $\alpha>1$ and all $n\geq 2.$ The estimate (\ref{E1}) provides in general more information than the estimate (\ref{E2}). However, the discrepancy error in (\ref{E1}) is in general larger than in (\ref{E2}). We will use (\ref{E1}) when $r$ is small compared to $q$ and (\ref{E2}) when $r$ is large.
		
		Now we need to consider the sum
		\[
		\frac{\psi'(q)}{q}\sum^{l\leq \omega \log q}_{l\geq 0}\sum_{r|q}r \sum_{k:1\leq 2^k\leq q^{1/2}}S_{k,l,r}(q).
		\]
		Let us first use the Estimate (\ref{E2}). In the sum, we only need to consider the case when Condition (\ref{C2}) is not satisfied, i.e. we have the following condition
		\begin{align}
		r<2^{\frac{k+4-l-\omega k}{1+\alpha}}\frac{1}{(\log q)^{1/(1+\alpha)}(\log\log q)^{2/(1+\alpha)}} q^{\frac{\omega+\alpha}{1+\alpha}}. \label{C2'}
		\end{align}
		The second term in (\ref{E2}) sums up to
		\begin{align}
		&\ll \frac{\psi'(q)}{q}\sum_{r|q}r\sum_{k,l} \left(\frac{3q}{2^k r}\right)^{\epsilon'}&\nonumber\\
		&\ll \psi'(q)(\log q)^2\sum_{r|q} \left(\frac{r}{q}\right)^{1-\epsilon'}=\psi'(q)o(1)&\label{I}
		\end{align}
		For the last line, we have used the condition (\ref{C2'}) and the fact that the number of divisors function $d(q)=o(q^{\delta})$ for all $\delta>0$. For the first term in (\ref{E2}), it sums up to
		\begin{align*}
		&\ll \frac{\psi'(q)}{q}\sum_{r|q}r\sum_{k,l}\frac{q^\omega q} {r^2}&\\
		&\ll \psi'(q)q^\omega (\log q)^2\sum_{r|q}\frac{1}{r}.&
		\end{align*}
		Thus, if we impose the condition $r\geq q^{2\omega}$ in the sum, we see that
		\begin{align}
		\psi'(q)q^\omega (\log q)^2\sum_{\substack{r|q\\r\geq q^{2\omega}}}\frac{1}{r}\leq \psi'(q)q^\omega (\log q)^2 d(q)q^{-2\omega}=\psi'(q)o(1).\label{II}
		\end{align}
		Now, we are left with the terms for which $r<q^{2\omega}.$ We will use the Estimate (\ref{E1}). First, we deal with the error term (the second term), which sums up to
		\begin{align*}
		&\ll \frac{\psi'(q)}{q}\sum_{\substack{r|q\\r<q^{2\omega}}}r\sum_{k,l}r\left(\frac{3q}{2^k r}\right)^{\epsilon}&\\
		&\ll \psi'(q)(\log q)\sum_{\substack{r|q\\r<q^{2\omega}}}\frac{r^{2-\epsilon}}{q^{1-\epsilon}}&\\
		&\ll \psi'(q)(\log q) d(q) q^{2\omega(2-\epsilon)-(1-\epsilon)}.&
		\end{align*}
		We can choose $\omega$ to be small enough so that $2\omega(2-\epsilon)-(1-\epsilon)<0.$\footnote{Here is the first condition on $\omega$.} Then we have
		\begin{align}
		\psi'(q)(\log q) d(q) q^{2\omega(2-\epsilon)-(1-\epsilon)}=\psi'(q)o(1).\label{III}
		\end{align}
		The last step is to add up the contributions for the first term of (\ref{E1}),
		\begin{align}
		\frac{\psi'(q)}{q}\sum^{l\leq \omega \log q}_{l\geq 0}\sum_{\substack{r|q\\ r<q^{2\omega}}}r \sum_{k:1\leq 2^k\leq q^{1/2}}8\frac{2^{k+1}}{\log q(\log\log q)^2} \frac{q^{\omega}}{2^l2^{\omega k}}\times\frac{1}{r} \times \frac{2^{l+1}2^{(k+1)\omega}}{q^{\omega}}\frac{3q}{2^k r}\nonumber\\
		\leq C\frac{\psi'(q)}{\log q(\log\log q)^2} \log q \sum_{r|q} \frac{\log r}{r}\nonumber\\
		=C\frac{\psi'(q)}{(\log\log q)^2} \sum_{r|q} \frac{\log r}{r},\label{IV}
		\end{align}
		where $C>0$ is a positive constant. Now we combine (\ref{I}),(\ref{II}),(\ref{III}),(\ref{IV}) to obtain (update the value for $C$ if necessary)
		\begin{align*}
		\sum_{\substack{1\leq q'<q\\\|q'\beta-\gamma'\|\in [{q'}^{-\omega},1)}} \frac{\psi'(q)}{q}\gcd(q',q)I_{\Delta(q,q')/\gcd(q,q')}(\{\gamma (q'-q)/\gcd(q,q')\})\\
		\leq C\left(\psi'(q)+\frac{\psi'(q)}{(\log\log q)^2}\sum_{r|q} \frac{\log r}{r}\right).
		\end{align*}
		This is the estimate for the terms in (\ref{SUM}) with $\psi'(q')q\geq \psi'(q)q'$.
		\subsection*{Case II: $\psi'(q')q\leq \psi'(q)q'$}
		In this case, we have
		\[
		\min\left\{\frac{\psi'(q')}{q'},\frac{\psi'(q)}{q}\right\}=\frac{\psi'(q')}{q'}.
		\]
		As in the first case, we are dealing with the following sum
		\begin{align*}
		\sum_{\substack{1\leq q'<q\\\|q'\beta-\gamma'\|\in [{q'}^{-\omega},1)}} \frac{\psi'(q')}{q'}\gcd(q',q)I_{\Delta(q,q')/\gcd(q,q')}(\{\gamma (q'-q)/\gcd(q,q')\})\\
		\leq \sum^{l\leq \omega \log q}_{l\geq 0}\sum_{r|q}r \sum_{k:1\leq 2^k\leq q^{1/2}}\sum_{\substack{q': \gcd(q',q)=r\\q'\in G^l\cap  D_k(q)}}\frac{\psi'(q')}{q'}I_{\Delta(q',q)/r}(\gamma(q'-q)/r).
		\end{align*}
		Observe that for $q'\in G^l\cap D_k(q),$ 
		\[
		\frac{\psi'(q')}{q'}\leq T_{k,l}=\min\left\{8\times \frac{2^{2(k+1)}}{q^2}\frac{q^\omega}{2^{\omega k}} \frac{1}{2^l}\frac{1}{\log q (\log\log q)^2}, \frac{\psi'(q)}{q}\right\}.
		\]
		Therefore we have
		\begin{align*}
		\sum_{\substack{q': \gcd(q',q)=r\\q'\in G^l\cap D_k(q)}}\frac{\psi'(q')}{q'}I_{\Delta(q',q)/r}(\gamma(q'-q)/r)\\
		\leq T_{k,l}\sum_{\substack{q': \gcd(q',q)=r\\q'\in G^l\cap D_k(q)}}I_{\Delta(q',q)/r}(\gamma(q'-q)/r).
		\end{align*}
		As we assumed for this case, we have
		\[
		\Delta(q',q)\leq 2 \psi'(q)q'.
		\]
		Then for $q'\in G^l\cap D_k(q),$ we have
		\[
		\Delta(q',q)\leq 2 \psi'(q) \frac{q}{2^k}.
		\]
		The following steps will be very similar to those in (Case I). Again, we need to estimate 
		\begin{align*}
		S_{k,l,r}(q)=\sum_{\substack{q': \gcd(q',q)=r\\q'\in G^l\cap D_k(q)}}I_{\Delta(q',q)/r}(\gamma(q'-q)/r)\\
		\end{align*}
		Again, as $\gamma$ is not Liouville, we have $S_{k,l,r}=0$ if
		\[
		2\psi'(q)\frac{q}{2^kr}\leq \left(\frac{q}{r}\right)^{-\alpha}.
		\]
		Therefore, we can assume that (negating the above)
		\begin{align}
		r^{1+\alpha}< 2\psi'(q)q^{\alpha+1}\frac{1}{2^k}\leq  \frac{2}{q\log q (\log\log q)^2} \times q^{\omega}\times q^{\alpha+1}\frac{1}{2^k}\leq 2q^{\alpha+\omega}.\label{C2''}
		\end{align}
		As $(\gamma,r\beta)$ is a Diophantine pair, we obtain the following estimate which is similar to (\ref{E1}),
		\begin{align}
		S_{k,l,r}(q)\leq \frac{2\psi'(q)q}{2^kr} \frac{2^{l+1+\omega (k+1)}}{q^\omega} \frac{3q}{2^k r}+C r\left(\frac{q}{2^k r}\right)^{\epsilon}.\label{E1'}
		\end{align}
		Similarly, we also have the following analogue of (\ref{E2}),
		\begin{align}
		S_{k,l,r}(q)\leq \frac{2\psi'(q)q}{2^k r}\frac{3q}{2^k r}+C \left(\frac{q}{2^k r}\right)^{\epsilon}.\label{E2'}
		\end{align}
		Our aim is to estimate
		\[
		\sum_{r|q}r\sum_{k,l} T_{k,l} S_{k,l,r}(q).
		\]
		As in (Case I), we need to switch between (\ref{E1'})(\ref{E2'}) according to the value of $r.$ First, let us observe that the error term (second term) in (\ref{E2'}) creates no significant contribution (with the condition (\ref{C2''})),
		\begin{align*}
		\sum_{r|q}r\sum_{k,l}T_{k,l} \left(\frac{q}{2^k r}\right)^{\epsilon}\ll \psi'(q)(\log q)^2\sum_{r|q} \left(\frac{r}{q}\right)^{1-\epsilon}\ll \psi'(q)(\log q)^2d(q) q^{-\rho}=\psi'(q)o(1),
		\end{align*}
		where $\rho>0$ is a positive number depending on $\alpha,\omega,\epsilon.$ Let us now consider the first term in (\ref{E2'}),
		\begin{align*}
		\sum_{r|q}r\sum_{k,l}T_{k,l}\frac{2\psi'(q)q}{2^k r}\frac{3q}{2^k r}\ll \psi'(q)q^\omega (\log q)^2\sum_{r|q}\frac{1}{r}.
		\end{align*}
	We will choose to use (\ref{E2'}) only when $r\geq q^{2\omega}.$ In this way, the above is again
		\[
		\psi'(q)o(1).
		\]
		We have thus showed that under the condition \ref{C2''}) and $r\geq q^{2\omega},$ 
		\begin{align}
		\sum_{\substack{r|q\\ r\geq q^{2\omega}}}r\sum_{k,l} T_{k,l} S_{k,l,r}(q)=\psi'(q)o(1).\label{I'}
		\end{align}
		From now on, we will assume $r<q^{2\omega}.$ Next, we see that the error term of (\ref{E1'}) (the second term above) sums up to
		\begin{align}
		\ll \sum_{\substack{r|q\\ r<q^{2\omega}}} r \sum_{k,l} T_{k,l} \times  r\left(\frac{q}{2^k r}\right)^{\epsilon}\nonumber\\
		\ll \psi'(q)(\log q)^2\sum_{\substack{r|q\\ r<q^{2\omega}}} \frac{r^2}{q}\left(\frac{q}{2^k r}\right)^{\epsilon}\nonumber\\
		\ll \psi'(q)(\log q)^2d(q)q^{-E},\label{II'}
		\end{align}
		where the exponent $E$ is (recall that $r<q^{2\omega}$)
		\[
		E=1-\epsilon-2\omega(2-\epsilon).
		\]
		Now $\epsilon$ is a fixed positive number. We have met this condition just above the Estimate (\ref{III}). Our freedom is to choose $\omega\in (0,1).$ We need to choose $\omega$ to be sufficiently small so that $E>0.$\footnote{This is the second and the last condition on $\omega$.} This is certainly possible. Finally, the main term (the first term) of (\ref{E1'}) sums up to
		\begin{align*}
		\leq \sum_{\substack{r|q\\ r<q^{2\omega}}} r \sum_{k,l} 48\times \frac{2^{2(k+1)}}{q^2}\frac{q^\omega}{2^{\omega k}} \frac{1}{2^l}\frac{1}{\log q (\log\log q)^2} \times  \frac{\psi'(q)q}{2^kr} \frac{2^{l+1+\omega (k+1)}}{q^\omega} \frac{q}{2^k r}\\
		\leq \frac{1000\psi'(q)}{\log q(\log\log q)^2} \sum_{r|q}r\sum_{k,l} \frac{1}{r^2}
		\end{align*}
		The sum of $k,l$ gives a $O(\log r\log q)$ factor, where the implied constant is absolute. Thus, we see that for a constant $C>0,$
		\begin{align}
		\frac{1000\psi'(q)}{\log q(\log\log q)^2} \sum_{r|q}r\sum_{k,l} \frac{1}{r^2}\leq \frac{C\psi'(q)}{(\log\log q)^2} \sum_{r|q} \frac{\log r}{r}.\label{III'}
		\end{align}
		We can combine (\ref{I'}), (\ref{II'}), (\ref{III'}) and obtain the required estimate for (\ref{SUM}) in (Case II). From here, the proof is finished.
	\end{proof}
	
	\subsection{Weakening the Diophantine condition: proof of Lemma \ref{Lemma3}}\label{W}
	In the above proof, the Diophantine condition for $(\gamma,\beta)$ is needed in the following way. First, we need $\gamma$ to be not Liouville to deduce Estimates (\ref{E2}), (\ref{E2'}). Next, together with Condition (\ref{C2'}) we deduced that the sum for $r\geq q^{2\omega}$ gives $\psi'(q)o(1).$ To deal with the sum for $r<q^{2\omega},$ we need to use the condition that $(\gamma,\beta)$ is a Diophantine pair in order to have (\ref{E1}), (\ref{E1'}).
	
	There is still leeway for us. In order to see it, let us examine the Diophantine property of the pair $(\gamma,\beta)$ more precisely. Recall that $\sigma(.)$ is the function taking integer variables and positive values so that $\sigma(N)$ is the best Diophantine exponent for $(\gamma,\beta)$ up to height $N.$ That is to say, $\sigma(N)$ is the infimum of all numbers $\sigma>0$ such that for all $-N\leq k_1,k_2\leq N$ with $k_1,k_2$ not both zeros,
	\[
	\|k_1\gamma+k_2\beta\|\geq \max\{k_1,k_2\}^{-\sigma}
	\] 
	We recall the following result by Erd\H{o}s-Tur\'{a}n-Koksma, \cite[Theorem 1.21]{DT97}. This result holds for $d$-dimensional torus rotations with any $d\geq 1$. In this paper, we only use the case when $d=1$ or $2$.
	\begin{theorem}[ETK]\label{ETK}
		Let $\alpha$ be an irrational number. Let $H,N$ be positive integers. Then for each interval $I\subset [0,1]$,
		\[
		|E^{I}_{\alpha}(N)|\leq 9N \left(\frac{1}{H}+\sum_{0<|k|\leq H}\frac{2}{||k|+1|}\frac{2}{N\|k\alpha\|}\right).
		\]

		Let $(\alpha,\beta)$ be a pair of numbers. Let $H, N$ be positive integers. Then for each rectangle  $I\subset [0,1]^2$ whose sides are parallel to the coordinate axes,
		\[
		|E^{I}_{\alpha,\beta}(N)|\leq 9N\left(\frac{1}{H}+\sum'_{k_1,k_2}\frac{4}{(|k_1|+1)(|k_2|+1)}\frac{2}{N \|k_1\alpha+k_2\beta\|}\right),
		\]
		where $\sum'_{k_1,k_2}$ is the sum of pairs of integers $-H\leq k_1,k_2\leq H$ such that at least one of $k_1,k_2$ is not zero.
	\end{theorem}
	Here, the factor $4/(|k_1|+1)(|k_2|+1)$ can be replaced with
	\[
	1/\max\{1,|k_1|\}\max\{1,|k_2|\},
	\]
	which will only provides us with some marginal improvements. We apply Theorem \ref{ETK} for $(\gamma,\beta).$ We see that
	\begin{align*}
	|E^I_{\gamma,\beta}(N)|&\leq 9N\left(\frac{1}{H}+\sum'_{k_1,k_2}\frac{4}{(|k_1|+1)(|k_2|+1)}\frac{2}{N \|k_1\gamma+k_2\beta\|}\right)&\\
	&\leq 9N\left(\frac{1}{H}+\frac{8}{N}\sum'_{k_1,k_2}\frac{1}{(|k_1|+1)(|k_2|+1)}\frac{1}{ \max\{k_1,k_2\}^{-\sigma(N)}}\right)&\\
	&\leq 9N\left(\frac{1}{H}+\frac{8000\sigma(N)}{N}(\log H)H^{\sigma(N)}\right).&
	\end{align*}
	The factor $8000$ on the last line is by no means the optimal one. Now we see that for an absolute constant $C>0,$
	\[
	9N\left(\frac{1}{H}+\frac{8000\sigma(N)}{N}(\log H)H^{\sigma(N)}\right)\leq C N^{\sigma(N)/(\sigma(N)+1)}(\log N)^{1/(\sigma(N)+1)}
	\]
	if we choose $H$ to be the smallest positive integer with
	\[
	\frac{1}{H}\leq \frac{8000\sigma(N)}{N}(\log H)H^{\sigma(N)}.
	\]
	Thus, we have obtained the following discrepancy estimate for the irrational rotation with parameter $(\gamma,\beta),$
	\begin{align}
	|E^I_{\gamma,\beta}(N)|\leq C_1N^{\sigma(N)/(\sigma(N)+1)}(\log N)^{1/(\sigma(N)+1)}.\label{R1}
	\end{align}
	Recall that we also have the following condition for all $0\leq |k_1|,|k_2|\leq N$ except $k_1=k_2=0,$
	\[
	\|k_1\gamma+k_2\beta\|\geq \max\{k_1,k_2\}^{-\sigma(N)}.
	\]
	In particular, if we let $k_2=0,$ then we have for all $1\leq |k_1|\leq N,$
	\begin{align}
	\|k_1\gamma\|\geq k_1^{-\sigma(N)}.\label{R2a}
	\end{align}
	Then a further use of Theorem \ref{ETK} gives us that
	\begin{align}
	|E^I_{\gamma}(N)|\leq C_1 N^{\sigma(N)/(\sigma(N)+1)}.\label{R2b}
	\end{align}
	Of course, if $(\gamma,\beta)$ is a Diophantine pair, then $\sigma(.)$ is bounded. Then from the above (\ref{R1}), (\ref{R2a}), (\ref{R2b}) we can deduce (\ref{E1}), (\ref{E1'}), (\ref{E2}), (\ref{E2'}). Here is where the few extra rooms are. We now prove Lemma \ref{Lemma3}.

	\begin{proof}[Proof of Lemma \ref{Lemma3}]
		We include all the notations in the proof of Lemma \ref{Lemma}. We want to insert some changes to the arguments in order to deduce Lemma \ref{Lemma3}. Let us concentrate on (Case I). Observe that we still have Estimates (\ref{E1}), (\ref{E2}) and Condition (\ref{C2}). The difference here is that the exponents $\epsilon,\epsilon',\alpha$ are no longer constants. They need to be chosen according to $q.$ In fact, here we can choose
		\[
		\epsilon(q)=\epsilon'(q)=(1+o(1))\frac{\sigma(q)}{\sigma(q)+1}
		\]
		and
		\[
		\alpha(q)=\sigma(q).
		\]
		Here the $o(1)$ term is introduced to absorb the power of $\log N$ factor in (\ref{R1}) and can be made more explicit. Next, in order to have the Estimate (\ref{I}), it is sufficient to have
		\begin{align}
		(\log q)^2 d(q) \frac{1}{q^{\rho(q)}}=o(1)\label{Ic}
		\end{align}
		where 
		\[
		\rho(q)=\frac{0.5-\omega(q)}{(\sigma(q)+1)^2}.
		\]
		Again, we may fix a small enough $\omega(q)$ to start with. Observe that (the maximal order of the divisor function)\footnote{The use of the maximal order of the divisor function is an overkill. In most situations, one only need to use the averaged order. For example, if $\psi$ is supported on where $d(q)=O(\log q),$ then we can improve this estimate significantly.}
		\[
		\log d(q)\ll \log q/\log\log q.
		\]
		Thus we can achieve (\ref{Ic}) if
		\[
		(\log q)^2d(q)=o(q^{\rho(q)}).
		\]
		This can be achieved by requiring that
		\[
		\sigma(q)\leq C (0.5-\omega(q))^{1/2}(\log\log q)^{1/2}.
		\]
		where $C>0$ is an absolute constant. The above condition says that $q^{\rho(q)}$ is not too small.
		
		After re-establishing (\ref{I}), we need to consider (\ref{II}). For (\ref{II}) to hold, it is enough to have
		\[
		(\log q)^2 d(q)/q^\omega=o(1).
		\]
		This can be achieved by choosing $\omega$ to be $\gg 1/(\log \log q)$ where the implied constant is absolute. This says that $\omega$ cannot be too small. The critical issue occurs for Estimate (\ref{III}). We see that if $\epsilon$ is very close to one, then $\omega$ must be chosen to be very small because we should have
		\[
		2\omega(2-\epsilon)-(1-\epsilon)<0.
		\]
		Thus, we cannot maintain that $\omega$ is a fixed number. It must vary along $q$ as well! More precisely, we can choose
		\begin{align}
		\frac{C_1}{\log\log q}\leq\omega(q)\leq C_2/\sigma(q)\label{Omega}
		\end{align}
		where $C_1,C_2>0$ are constants. Since $\sigma(q)=O((\log\log q)^{1/2})$, the above requirement is possible to be achieved. Thus we have
		\[
		q^{2\omega(q)(2-\epsilon(q))-(1-\epsilon(q))}(\log q)^2 d(q)=o(1).
		\]
		In this way, we still have the Estimates (\ref{I}), (\ref{II}), (\ref{III}). The estimate (\ref{IV}) is the main term, we do not have further restriction to deduce it. However, observe that as we have already imposed the condition $r\leq q^{2\omega(q)}$, this restricts the range of $k$ in the sum because $2^k\leq r^2$. In return, assuming $\psi(q)=O(q^{-1}(\log q)^{-1}(\log\log q)^{-3}),$ we see that
		\begin{align}
		\frac{\psi'(q)}{q}\sum^{l\leq \omega(q) \log q}_{l\geq 0}\sum_{r|q}r \sum_{k:1\leq 2^k\leq q^{2\omega(q)}}8\frac{2^{k+1}}{\log q(\log\log q)^3} \frac{q^{\omega(q)}}{2^l2^{\omega(q) k}}\times\frac{1}{r} \times \frac{2^{l+1}2^{(k+1)\omega(q)}}{q^{\omega(q)}}\frac{3q}{2^k r}\nonumber\\
		\ll \frac{\psi'(q)\omega(q)}{(\log\log q)^{3}}\sum_{r|q}\frac{\log r}{r}. \label{IVc}
		\end{align}
		Similarly, the arguments in (Case II) can be reconsidered in exactly the same way. This finishes the proof.
	\end{proof}

	\section{Proofs of Theorems \ref{MAIN}, \ref{MAIN2}} \label{proofofmain}
	In this section, we will prove Theorems \ref{MAIN}, \ref{MAIN2}.
	
	\begin{proof}[Proof of Theorem \ref{MAIN}]
		First, the requirement that $(\gamma,\beta)$ is Diophantine implies that $\gamma,\beta$ are both irrational and non-Liouville numbers.
		Let $\omega$ be a small enough positive number which will be specified later. Recall the construction (\ref{psi}) for $\psi'=\psi'_{\beta,\gamma',\omega}.$ We see that
		\[
		W(\psi,\beta,\gamma,\gamma')\supset W(\psi',\gamma).
		\]
		Now we write $A$ for $A^{\psi',\gamma}.$ The proof now divides into two parts.
		\subsubsection*{Step 1: Restricting the support of $\psi'$}
		In this step, we show that it is possible to restrict the support of $\psi'$ a bit further (to a certain subset $A\subset\mathbb{N}$). The idea is that we have a upper bound condition $\psi(q)=O((q\log q(\log\log q)^2)^{-1}).$ Then we can use this upper bound to deduce another upper bound for $\psi'$ in such a way that 
		\[
		\sum_{q\in A^c}\psi'(q)<\infty.
		\]
		On the other hand, we also know that $\sum_{q=1}^\infty\psi'(q)=\infty.$ Thus we have
		\[
		\sum_{q\in A}\psi'(q)=\infty.
		\]
		This allows us to restrict the support of $\psi'$ to $A.$ We now supply the details.
		
		Recall that
		\[
		F(q)=\sum_{r|q} \frac{\log r}{r}.
		\]
		In this step, we show that it is possible to restrict $\psi'$ on integers $q$ with $F(q)=O((\log\log q)^2).$
		
		We analyse the values of $F(q)$ for $q$ with
		\[
		\|q\beta-\gamma'\|\in \left[\frac{1}{q^{\omega}},1\right].
		\]
		For each integer $l\geq 0,$ recall the set
		\[
		G^l=G_\omega^l=\left\{q:\|q\beta-\gamma'\|\in \left[\frac{2^l}{q^{\omega}},\frac{2^{l+1}}{q^{\omega}}\right]\right\}.
		\]
		Let $K>0$ be an integer and $Q>1000$ be another integer. Consider the following sum
		\begin{align*}
		\sum_{q\in G^l\cap [Q/2,Q]} F^K(q)=\sum_{r_1,\dots r_K\leq Q} \frac{\log r_1\log r_2\dots \log r_K}{r_1 r_2\dots r_K}\sum_{\substack{q: [r_1,\dots,r_K]|q\\q\in G^l\cap [Q/2,Q]}}1.
		\end{align*}
		By using Lemma \ref{Discrepancy bound} we see that
		\begin{align}
		\sum_{\substack{q: [r_1,\dots,r_K]|q\\q\in G^l\cap [Q/2,Q]}}1\leq \frac{Q}{[r_1,\dots,r_K]}\frac{2^{l+1+\omega}}{Q^{\omega}}+C[r_1,\dots,r_k] \left(\frac{Q}{[r_1,\dots,r_K]}\right)^{\epsilon}\label{E3}
		\end{align}
		For two constants $\epsilon,C>0$ which depend only on $\beta.$ Here the factor $[r_1,\dots,r_k]$ for the second term on the RHS is because we are considering the rotation with angle $\beta [r_1,\dots,r_k]$ rather than $\beta.$
		
		From now on, we pose the condition
		\begin{align}
		1-\omega-\epsilon>0. \label{C3}
		\end{align}
		This is always possible to achieve since $\epsilon<1.$ We will also need to apply Lemma \ref{Lemma}. From there, we also got an upper bound $\omega'_0$ for possible values of $\omega.$ Now we set $\omega_0$ to be the minimal between $(1-\epsilon)$ and $\omega'_0.$
		
		The first term on the RHS in (\ref{E3}) is relatively easy to handle and it will make the major contribution. We want to show that the second term would not contribute too much. Observe that if
		\[
		\frac{Q^{1-\omega}}{[r_1,\dots,r_K]}\geq  [r_1,\dots,r_k]\left(\frac{Q}{[r_1,\dots,r_K]}\right)^{\epsilon}
		\]
		then the RHS of (\ref{E3}) will be bounded from above by
		\[
		\left(1+\frac{C}{2^{l+1+\omega}}\right) \frac{Q}{[r_1,\dots,r_K]}\frac{2^{l+1+\omega}}{Q^{\omega}}\leq \left(1+C\right) \frac{Q}{[r_1,\dots,r_K]}\frac{2^{l+1+\omega}}{Q^{\omega}}.
		\]
		Otherwise, we have
		\begin{align}
		[r_1,\dots,r_K]\geq Q^{\frac{1-\omega-\epsilon}{2-\epsilon}}. \label{C4} 
		\end{align}
		We write $\epsilon'=\frac{1-\omega-\epsilon}{2-\epsilon}.$ Observe that
		\begin{align*}
		\sum_{\substack{r_1,\dots r_K\leq Q\\ [r_1,\dots,r_K]\geq Q^{\epsilon'}}} \frac{\log r_1\log r_2\dots \log r_K}{r_1 r_2\dots r_K}\sum_{\substack{q: [r_1,\dots,r_K]|q\\q\in G^l\cap [Q/2,Q]}}1\\
		\leq \frac{\log^K Q}{Q^{\epsilon'}} \sum_{r_1,\dots,r_K\leq Q} \sum_{\substack{q: [r_1,\dots,r_K]|q\\q\in G^l\cap [Q/2,Q]}}1\\
		=\frac{\log^K Q}{Q^{\epsilon'}}\sum_{q\in G^l\cap [Q/2,Q]} d^K(q).
		\end{align*}
		There is a constant $C_1>0$ such that for all $q\geq 1,$
		\[
		d(q)\leq 2^{C_1\log q/\log\log q}.
		\]
		Thus we see that
		\begin{align}
		\frac{\log^K Q}{Q^{\epsilon'}}\sum_{q\in G^l\cap [Q/2,Q]} d^K(q)\nonumber\\
		\leq \frac{\log^K Q}{Q^{\epsilon'}}2^{C'K\log Q/\log\log Q}\# G^l\cap [Q/2,Q].\label{I'1}
		\end{align}
		By using Lemma \ref{Discrepancy bound} we see that as $\epsilon<1-\omega$,
		\[
		\# G^l \cap [Q/2,Q]\leq \frac{2^{l+1+\omega}}{Q^{\omega}}Q+CQ^\epsilon\leq (2^{l+1+\omega}+C)Q^{1-\omega}.
		\]
		Next, we bound the first term on the RHS of (\ref{E3}),
		\begin{align}
		\sum_{r_1,\dots r_K\leq Q} \frac{\log r_1\log r_2\dots \log r_K}{r_1 r_2\dots r_K}\frac{Q}{[r_1,\dots,r_K]}\frac{2^{l+1+\omega}}{Q^{\omega}}\nonumber\\
		\leq Q^{1-\omega}2^{l+1+\omega} (C_2K^2)^K\label{II'1}
		\end{align}
		where $C_2>0$ is an absolute constant. Here, we have used the fact that
		\[
		\sum_{r\geq 1} \frac{\log r}{r^{1+K^{-1}}}=-\zeta'(1+K^{-1})= O(K^2)
		\]
		and that\footnote{This inequality is sharp. However, this is an overkill in the estimate of the sum $\sum_{r_1,\dots,r_K}$. For most of the tuples $(r_1,\dots,r_K),$ $[r_1,\dots,r_K]$ is in fact much larger than $(r_1\dots r_K)^{1/K}$. Intuitively speaking, very few tuples of integers have large GCD's.}
		\[
		[r_1,\dots,r_K]\geq (r_1\dots r_K)^{1/K}.
		\]
		Collecting \ref{I'1}),(\ref{II'1}) we see that
		\begin{align}
		\sum_{q\in G_{\omega}^l\cap [Q/2,Q]} F^K(q)\leq Q^{1-\omega}2^{l+1+\omega} (C_1K^2)^K+\frac{\log^K Q}{Q^{\epsilon'}}2^{C_1K\log Q/\log\log Q}(2^{l+1+\omega}+C)Q^{1-\omega}\label{RAW}
		\end{align}
		holds for all $Q\geq 1024,K,l\geq 1$ The above constants $C,C_1,C_2$ do not depend on $Q,K,l.$ We choose $K=K_Q$ as a function of $Q.$ Then as long as  $K\leq \epsilon'\log\log Q/2C_1$ we see that with another constant $C_3>0,$
		\begin{align}
		\sum_{q\in G^l\cap [Q/2,Q]} F^{K_Q}(q)\leq C_3 2^{l}Q^{1-\omega}((C_1K_Q^2)^{K_Q}+\log^{K_Q} Q/Q^{\epsilon'/2})\label{E4}.
		\end{align}
		Now we choose $K_Q=[\epsilon'\log\log Q/2C_1].$ From here we see that for $M$ with $\log M=8C_1/\epsilon'$,
		\begin{align}
		\#\{q: q\in G^l\cap [Q/2,Q], F(q)\geq C_1M (\log\log Q)^2\}\leq C_42^{l}Q^{1-\omega} \frac{1}{(\log Q)^4}, \label{E5}
		\end{align}
		where $C_4>0$ is a constant depending on both $\omega$ and $\epsilon'$ which in turn depends only on $\beta.$ From here we see that ($C_5>0$ is a different constant)
		\begin{align*}
		\sum_{\substack{q\in [Q/2,Q]\\F(q)\geq C_1M(\log\log q)^2}} \psi'(q)&\leq\sum_{0\leq l\leq \omega\log Q} \sum_{\substack{q\in [Q/2,Q]\cap G^l\\F(q)\geq C_1M(\log\log q)^2}} \frac{1}{q \log q (\log\log q)^2} \frac{q^{\omega}}{2^l}\\
		&\leq C_5\sum_{0\leq l\leq \omega \log Q} \frac{1}{Q \log Q (\log\log Q)^2} \frac{Q^{\omega}}{2^l} 2^l Q^{1-\omega}\frac{1}{(\log Q)^4}\\
		&\leq C_5\omega \frac{1}{(\log Q)^4 (\log\log Q)^2}.
		\end{align*}
		By considering $Q=2^k,k\geq 10$ we see that
		\[
		\sum_{F(q)\geq C_1M(\log\log q)^2} \psi'(q)<\infty.
		\]
		However, our assumption is that
		\[
		\sum_{q}\psi'(q)=\infty.
		\]
		Therefore we can make the assumption that $\psi'$ is supported on where $F(q)\leq C_1M (\log\log q)^2.$
		
		\subsubsection*{Step 2: Use Lemma \ref{Lemma}}
		
		Now we use Lemma \ref{Lemma} and see that
		\[
		\sum_{1\leq q'<q} |A_q\cap A_{q'}|=O\left(\psi'(q)+\frac{\psi'(q)}{(\log\log q)^2}F(q)\right)+4(1+C_0/(2H))\psi'(q)\sum_{1\leq q'< q} \psi'(q').
		\]
		As we assumed that $\psi'(q)>0$ only when $F(q)\leq C'M (\log\log q)^2,$ the above estimate can be made one step further,
		\begin{align*}
		\sum_{1\leq q'<q} |A_q\cap A_{q'}|&=O\left(\psi'(q)+C'M\psi'(q)\right)+16\psi'(q)\sum_{1\leq q'< q} \psi'(q')\\&= O(\psi'(q))+4(1+C_0/(2H))\psi'(q)\sum_{1\leq q'< q} \psi'(q').
		\end{align*}
		From here we see that (as $\sum_q \psi'(q)=\infty$)
		\[
		\frac{(\sum_{q\leq Q} |A_q|)^2}{\sum_{q,q'\leq Q} |A_q\cap A_{q'}|}\geq\frac{(\sum_{q=1}^Q 2\psi'(q))^2}{4(1+C_0/(2H))(\sum_{q=1}^Q\psi'(q))^2+O(\sum_{q=1}^Q \psi'(q))}\geq \frac{1}{1+\frac{C_0}{2H}+o(1)}.
		\]
		Then by Lemma \ref{Borel} we see that
		\[
		|\limsup_{q\to\infty} A_q|\geq \frac{1}{1+\frac{C_0}{2H}}
		\]
		This finishes the proof as we can choose $H$ to be arbitrarily large.
		
	\end{proof}
	
	\begin{proof}[Proof of Theorem \ref{MAIN2}]
		In the proof of Theorem \ref{MAIN}, we obtained (\ref{E3}) from the non-Liouville condition for $\beta.$ This estimate finally gives (\ref{E4}) and (\ref{E5}). Here, we need to choose $\omega,\epsilon$ to change along $q.$ For example, the construction of $G_\omega^l$ is now \[
		G_\omega^l=\left\{q:\|q\beta-\gamma'\|\in \left[\frac{2^l}{q^{\omega(q)}},\frac{2^{l+1}}{q^{\omega(q)}}\right]\right\}
		\]
		and $\epsilon(q)\leq\sigma_\beta(q)/(\sigma_\beta(q)+1)$, where $\sigma_\beta(q)$ is the Diophantine exponent for $\beta$ at height $q.$ In particular, we have $\sigma_\beta(q)\leq \sigma_{(\gamma,\beta)}(q)=\sigma(q)=O(\log\log\log q).$ Also, in the definition $\psi'$, $\omega$ is a function as well. In this way, $\psi'$ is not zero only when $\|q\beta-\gamma'\|\gg q^{-\omega(q)}.$ Since $\omega(q)$ decays to $0,$ the support of $\psi'$ now is smaller.
		
		Now the estimate (\ref{RAW}) holds with all the exponents $\omega,\epsilon'$ being a function of $q.$ However, what we really need is the values $\omega(q),\epsilon(q),\epsilon'(q)$ for $q\in [Q,2Q].$ For this reason, it is more convenient to have concrete function forms for them. More precisely, we choose a small number $c>0$ and define
		\[
		\epsilon(q)=\frac{(\log\log\log q)^{1/2}}{(\log\log\log q)^{1/2}+1}
		\]
		\[
		\epsilon'(q)=\frac{1-\omega(q)-\epsilon(q)}{2-\epsilon(q)}.
		\]
		We fix $\omega(q)=c(\log\log\log q)^{-1/2}$ so that $\epsilon'(q)>0.$ The problem is that $\epsilon'(q)\to 0$ as $q\to\infty.$ This will cause $M$ (which is now also depending on $q$) to be too large. For this reason we need to treat the Estimate (\ref{RAW}) more carefully. Let us consider two sequences $K_Q, H_Q, Q\geq 1.$ Then we see that by (\ref{RAW})
		\begin{align*}
		&\#\{q:q\in G^l_\omega\cap [Q,2Q],F(q)\geq H_Q\}&\\\leq &\frac{1}{H^{k_Q}_Q}\sum_{q\in G_{\omega}^l\cap [Q/2,Q]} F^{K_Q}(q)&\\
		\leq & \frac{Q^{1-\omega(2Q)}2^{l+1+\omega(Q)} (C'K_Q^2)^{K_Q}+\frac{\log^{K_Q} Q}{Q^{\epsilon'(2Q)}}2^{C'K_Q\log Q/\log\log Q}(2^{l+1+\omega(Q)}+C)Q^{1-\omega(2Q)}}{H^{K_Q}_Q}.&
		\end{align*}
		Now we want to achieve the following two asymptotics,
		\[
		\left(\frac{K^2_Q}{H_Q}\right)^{K_Q}=O(1/(\log Q)^4),
		\]
		\[
		\frac{\log^{K_Q} Q}{Q^{\epsilon'(2Q)}}2^{C'K_Q\log Q/\log\log Q}\frac{1}{H_Q^{K_Q}}=O(1/(\log Q)^4).
		\]
		The first asymptotic is  satisfied if we choose
		\[
		H_Q\geq K^2_Q (\log Q)^{4/K_Q}=e^{2\log K_Q+4\log\log Q/K_Q}.
		\]
		The second asymptotics is satisfied if we choose
		\begin{align*}
		H_Q\geq \log Q\times 2^{C'\log Q/\log\log Q}\times (\log Q)^{4/\log K_Q}/Q^{\epsilon'(2Q)/K_Q}\\=2^{\log\log Q+C'\log 2 \log Q/\log\log Q+4\log\log Q/K_Q}e^{-\epsilon'(2Q)\log Q/K_Q}.
		\end{align*}
		To satisfy both of the above two inequalities, it is enough to take
		\[
		K_Q=\frac{\epsilon'(2Q)\log Q}{\log\log Q+\frac{C'\log 2 \log Q}{\log\log Q}}
		\]
		and
		\[
		H_Q=e^{2\log K_Q+4\log\log Q/K_Q}.
		\]
		This makes
		\[
		K_Q\asymp \log\log Q/(\log\log\log Q)^{1/2},H_Q=o((\log\log Q)^{3})
		\]
		and
		\begin{align}
		\#\{q:q\in G^l_\omega\cap [Q,2Q],F(q)\geq (\log\log Q)^{3}\}\ll 2^{l}Q^{1-\omega(2Q)}\frac{1}{(\log Q)^4}. \label{E51'}
		\end{align}
		So we have found a replacement of (\ref{E5}). Then the rest of the argument after (\ref{E5}) in the proof of Theorem \ref{MAIN} can be used.  Thus we can assume that $\psi'$ is supported on the integers $q$ with $F(q)\ll (\log\log q)^3.$ After this step, we can use the first part of Lemma \ref{Lemma3} to deduce the result.
	\end{proof}
	
	\section{Monotonic approximation function, proofs of Theorems \ref{MAIN MONNTON}, \ref{MAIN MONO2}}

	As before, we will first show the proof for Theorem \ref{MAIN MONNTON} in detail and then illustrate how to add new arguments to proof Theorems \ref{MAIN MONO2}.
	\begin{proof}[Proof of Theorem \ref{MAIN MONNTON}]
		First, we want to compare $\psi(q)$ with $(q\log q)^{-1}.$ Suppose that 
		\[
		\psi(q)\geq \frac{1}{q\log q}
		\]
		for only finitely many integers $q.$ Then we can actually assume that
		\[
		\psi(q)\leq \frac{1}{q\log q}
		\]
		for all $q\geq 2.$ We call this case to be the finite case. Otherwise, we have the infinite case. The central idea for proving the result for those two cases are very similar but we need to treat them in different ways.
		
		\subsection*{The finite case}
		Suppose that $\psi(q)\leq (q\log q)^{-1}.$ Then according to Remark \ref{Remark}, we would want to restrict $\psi$ to where $F(q)\leq H'$ for a suitable constant $H'>0.$ We could not do this for proving Theorem \ref{MAIN}. The monotonicity of $\psi$ will play a crucial role.
		
		First we choose a positive number $\omega$ which is small enough so that the conclusion of Remark \ref{Remark} hold. We also need to use (\ref{E4}) in the proof of Theorem \ref{MAIN} with $K_Q=1.$ In order to achieve this, we need to choose $\omega$ to be small enough.
		
		As before, our goal now is to show that
		\begin{align}
		\sum_{F(q)\leq H'}\psi'(q)=\sum_{\substack{q:\|q\beta-\gamma'\|\in [q^{-\omega},1]\\ F(q)\leq H'}}\frac{\psi(q)}{\|q\beta-\gamma'\|}=\infty \label{*}
		\end{align}
		with a suitable positive number $H'.$ 
		We make use of (\ref{E4}) with $K_Q=1$ and see that for all $Q>1024,l\geq 0$
		\begin{align}
		\sum_{q\in G_{\omega}^l\cap [Q/2,Q]} F(q)\leq C 2^{l}Q^{1-\omega}(C+\log Q/Q^{\epsilon'/2})\leq C2^lQ^{1-\omega}\label{E4'}
		\end{align}
		where $C>0$ is a constant depending on $\gamma,\beta$ and $G_\omega^l$ is the same as defined in the previous section,
		\[
		G^l=G_\omega^l=\left\{q:\|q\beta-\gamma'\|\in \left[\frac{2^l}{q^{\omega}},\frac{2^{l+1}}{q^{\omega}}\right]\right\}.
		\]
		From Lemma \ref{Discrepancy bound} here we see that for some $C_1>0,$
		\[
		\#\{q\in G^l\cap [Q/2,Q]: F(q)>H'\}\leq \frac{C_1}{H'} 2^l Q^{1-\omega}.
		\]
		On the other hand with again Lemma \ref{Discrepancy bound}, we have for a constant $C_2>0$ depending on $\beta,\gamma,$
		\[
		\# G^l\cap [Q/2,Q]\geq C_2 2^l Q^{1-\omega}.
		\]
		In order to use Lemma \ref{Discrepancy bound} in above, $1-\omega$ must be larger than the discrepancy exponent of $\beta.$ This can be achieved by choosing $\omega$ to be small enough. As a result, we see that for each even number $Q>2048,$
		\begin{align}
		\sum_{\substack{q\in [Q/2,Q]\\F(q)\leq H'}} &\psi'(q) \nonumber\\
		&\geq \psi(Q/2)\sum_{0\leq l\leq \omega\log Q} \sum_{\substack{q\in [Q/2,Q]\cap G^l\\ F(q)\leq H'}} \frac{q^{\omega}}{2^l}\nonumber\\
		&\geq \psi(Q/2) \frac{Q^{\omega}}{2^{\omega}}\sum_{0\leq l\leq \omega\log Q} \frac{1}{2^l}\#\{q\in G^l\cap [Q/2,Q]: F(q)\leq H'\}\nonumber\\
		&\geq \psi(Q/2) \frac{Q^{\omega}}{2^{\omega}} \sum_{0\leq l\leq \omega\log Q} \frac{1}{2^l} 2^l Q^{1-\omega}(C_2-C_1/H').\label{E5'}
		\end{align}
		We now choose $H'$ to be large enough in a manner that depends on $C_1,C_2$ which in turn depends on $\beta,\gamma$ only, such that
		\[
		C_2-C_1/H'>0.5C_2.
		\]
		As a result, we see that
		\[
		\sum_{\substack{q\in [Q/2,Q]\\F(q)\leq H'}} \psi'(q)\geq 0.5\times 2^{-\omega}C_2\psi(Q/2) Q \log Q.
		\]
		Taking $Q=2^k$ for $k\geq 11$ we see that
		\[
		\sum_{\substack{q\in [2^{k-1},2^k]\\F(q)\leq H'}} \psi'(q)\geq 0.5\times 2^{-\omega}C_2\psi(2^{k-1}) 2^{k} k\geq 0.5\times 2^{-\omega}C_2 \psi(2^{k-1})2^{k-1} \log k.
		\]
		Notice that as $\psi$ is non-increasing,
		\[
		\psi(2^{k-1})2^{k-1} \log k\geq \sum_{2^{k-1}\leq q\leq 2^{k}} \psi(q) \log q.
		\]
		From here we see that
		\[
		\sum_{\substack{q\geq 2048\\ F(q)\leq H'}} \psi'(q)\geq 0.5\times 2^{-\omega}C_2\sum_{k\geq 10} \sum_{\substack{q\in [2^{k-1},2^k]\\ F(q)\leq H'}} \psi'(q)\geq 0.5\times 2^{-\omega}C_2\sum_{q\geq 2048} \psi(q)\log q=\infty.
		\]
		This establishes (\ref{*}). Then we see that we can further restrict $\psi'$ on integers $q$ with 
		\[
		F(q)\leq H'.
		\]
		Then we can perform the arguments in (Step 2) of the proof of Theorem \ref{MAIN} to conclude that
		\[
		|W(\psi',\gamma)|=1.
		\]
		This finishes the proof for the finite case. Before we continue the proof. Let us first see how the above arguments extend to deal with the case when $\beta$ is Liouville. As in the proof of Theorem \ref{MAIN MONO2}, we need to consider $\omega$ as a function of $q$ rather than a constant. More precisely, we need
		\[
		\omega(q)\leq 1/(\sigma_\beta(q)+1).
		\]
		We will return to this discussion later. 
		\subsection*{The infinite case}
		In this case, there are infinitely many $q$ with
		\[
		\psi(q)\geq \frac{1}{q\log q}.
		\]
		Then as $\psi$ is non-increasing we can find infinitely many $Q$ such that
		\[
		\psi(q)\geq \frac{1}{2q\log q}
		\]
		for $q\in [Q/2,Q].$ For such a $Q$, consider the approximation function
		\[
		\psi_Q(q)=\frac{1}{2Q\log Q}
		\]
		when $Q/2\leq q\leq Q$ and $\psi_Q(q)=0$ otherwise. As there are infinitely many such $Q,$ we can choose $4096\leq Q_1<Q_2<Q_3<\dots$ such that $2Q_i\leq Q_{i+1}$ for $i\geq 1.$ Consider the new approximation function
		\[
		\psi_*=\sum_{i\geq 1} \psi_{Q_i}.
		\]
		This approximation function satisfies $\psi\geq \psi_*.$ We define
		\[
		\psi'(q)=\frac{\psi_*(q)}{\|q\beta-\gamma'\|}
		\]
		if $\|q\beta-\gamma'\|\in [q^{-\omega},1]$ and otherwise $\psi'(q)=0.$
		By performing the same arguments as in the finite case, we use Estimate (\ref{E5'}) and see that for each $i\geq 1,$
		\begin{align}
		\sum_{\substack{q\in [Q_i/2,Q_i]\\ F(q)\leq H'}} \psi'(q)\geq 0.5\times 2^{-\omega}C_2\psi(Q_i/2) Q_i \log Q_i\\=2^{-\omega}C_2 \frac{1}{2Q_i\log Q_i} Q_i \log Q_i=2^{-\omega-1}C_2.\label{E6'}
		\end{align}
		Thus we see that
		\[
		\sum_{q\geq 2048: F(q)\leq H'}\psi'(q)=\infty.
		\]
		The rest of the argument will be the same as in the finite case and we conclude that
		\[
		|W(\psi',\gamma)|=1.
		\]
		This finishes the proof for the infinite case and from here we conclude the theorem.
	\end{proof}
	\begin{proof}[Proof of Theorem \ref{MAIN MONO2}]
		We include all notations from the proof of Theorem \ref{MAIN MONNTON}. First, we re-examine the proof of the finite case of Theorem \ref{MAIN MONNTON}. The benchmark function is no longer $1/q(\log q)$ but 
		\[
		\frac{(\log\log q)^{1/2}}{q\log q}.
		\]
		We saw that the arguments in the proof of Theorem \ref{MAIN MONNTON} still hold here but with $\omega$ being chosen to be a function such that
		\[
		\omega(q)\leq 1/(\sigma_\beta(q)+1).
		\]
		Since we have $\sigma_\beta(q)=O((\log\log q)^{1/2}),$ we see that it is possible to choose $\omega$ in a way that
		\[
		\omega(q)\asymp (1/\log\log q)^{1/2}.
		\]
		We can use the second part of Lemma \ref{Lemma3}. Thus we need to show that the approximation function (Note that $\omega$ is not a constant.It is a function.)
		\[
		\psi'=\psi'_{\beta,\gamma',\omega}
		\]
		is divergent on a subset where $F$ is uniformly bounded from above, i.e.
		\[
		\sum_{q: F(q)\leq H'} \psi'(q)=\infty.
		\]
		To do this, we need to use (\ref{E5'}) in the proof of Theorem \ref{MAIN MONNTON}. Here, $\omega$ need to be replaced by $\omega(Q).$ Then we see that
		\[
		\sum_{q\in [Q,2Q],F(q)\leq H'}\psi'(q)\gg \psi(Q/2)Q \log Q \omega(Q).
		\]
		Then we see that
		\[
		\sum_{F(q)\leq H'}\psi'(q)\gg \sum_{q} \psi(q)\log q \frac{1}{(\log\log Q)^{1/2}}=\infty.
		\]
		This finishes the proof for the finite case. The infinite case can be treated similarly. We examine each dyadic interval where $\psi(q)\geq \frac{(\log\log q)^{1/2}}{q\log (q)}.$ Then (E6') needs to be changed to
		\begin{align*}
		\sum_{q\in [Q_i/2,Q_i], F(q)\leq H'} \psi'(q)\gg \frac{(\log\log Q_1)^{1/2}}{Q_i\log Q_i} Q_i \log Q_i \omega(Q_i)\gg 1.
		\end{align*}
		From here the proof finishes.
	\end{proof}

	\section{A doubly metric result}
	In order to prove Corollary \ref{Coro2}, we need the following standard result.
	
	\begin{lemma}\label{LMA}
		Let $\gamma$ be an irrational and non-Liouville number. Then for Lebesgue almost all $\beta\in\mathbb{R},$ $(\gamma,\beta)$ is Diophantine.
	\end{lemma}
	\begin{proof}
		First, it is clear that Lebesgue almost all numbers are $\mathbb{Q}$-linearly independent with respect to $\gamma.$ Next we let $H'$ to be a large positive number and we consider the set of $\beta\in [0,1]$ such that
		\[
		\|q_1\gamma+q_2\beta\|\leq \max\{|q_1|,|q_2|\}^{-H'}
		\] 
		for infinitely many integer pairs $(q_1,q_2)$ with $q_1q_2\neq 0.$ For such an integer pair $(q_1,q_2)$ we construct the following set
		\[
		A_{q_1,q_2}=\{\beta\in [0,1]: \|q_1\gamma+q_2\beta\|\leq \max\{|q_1|,|q_2|\}^{-H'}\}.
		\]
		It is then possible to see that (if we view $[0,1]$ as $\mathbb{T}$) $A_{q_1,q_2}$ is a union of $q_2$ many intervals (possibly with overlaps) of length
		\[
		\frac{2}{q_2 \max\{|q_1|,|q_2|\}^{H'}}.
		\]
		The Lebesgue measure of $A_{q_1,q_2}$ is then at most
		\[
		2\max\{|q_1|,|q_2|\}^{-H'}.
		\]
		Observe that by choosing $H'$ to be larger than $2$ we see that,
		\[
		\sum_{|q_1|,|q_2|\geq 1} \max\{|q_1|,|q_2|\}^{-H'}=\sum_{k\geq 1} \sum_{1\leq |q_1|\leq k} 2 k^{-H'}=\sum_{k\geq 1} 4k^{-H'+1}<\infty,
		\]
		Then the convergence part of Borel-Cantelli lemma implies that
		\[
		\left|\limsup_{|q_1|,|q_2|>0} A_{q_1,q_2}\right|=0.
		\]
		From here, we see that for Lebesgue almost all numbers $\beta\in [0,1]$ there are at most finitely many solutions ($q_1,q_2$) to
		\[
		\|q_1\gamma+q_2\beta\|\leq \max\{|q_1|,|q_2|\}^{-H'}.
		\] 
		Thus there is a constant $c>0$ with
		\[
		\|q_1\gamma+q_2\beta\|\geq c\max\{|q_1|,|q_2|\}^{-H'}
		\] 
		for all $|q_1|, |q_2|>0.$ We need to consider the situation when $q_1$ or $q_2$ is zero. We know that $\gamma$ is non-Liouville by assumption and we can also assume that $\beta$ is non-Liouville since non-Liouville numbers are Lebesgue typical. By redefining the constants $c,H'$ if necessary we see that for $|q_1|\geq 1, |q_2|\geq 1$
		\[
		\|q_1\gamma\|\geq c |q^{-H'}_1|
		\]
		as well as
		\[
		\|q_2\beta \|\geq c |q^{-H'}_2|.
		\]
		This implies that $(\gamma,\beta)$ is Diophantine. It is simple to extend the range for $\beta$ from $[0,1]$ to $\mathbb{R}.$ From here the proof finishes.
	\end{proof}
	
	\begin{proof}[Proof of Corollary \ref{Coro2}]
		Since $\gamma$ is non-Liouville, we can apply Theorem \ref{MAIN MONNTON} whenever $\beta$ and $(\beta,\gamma)$ are both Diophantine. From Lemma \ref{LMA} we also know that for Lebesgue almost all $\beta\in [0,1],$  the pair $(\gamma,\beta)$ is Diophantine. Thus $\beta$ is also non-Liouville and irrational. Therefore we see that Lebesgue almost all $\beta\in [0,1],$ both $\beta$ and $(\gamma,\beta)$ are Diophantine. We can then apply Theorem \ref{MAIN MONNTON} to those $\beta.$ We denote the set of such $\beta's$ as $G.$ Then we see that
		\[
		\|qx-\gamma\|\|q\beta-\gamma'\|<\psi(q)
		\]
		infinitely often for Lebesgue almost all $x$ whenever $\beta\in G.$ Then we can use Fubini's theorem to conclude Corollary \ref{Coro2}.
	\end{proof}

	\subsection*{Acknowledgements.}
	HY was financially supported by the University of Cambridge and the Corpus Christi College, Cambridge. HY has received funding from the European Research Council (ERC) under the European Union’s Horizon 2020 research and innovation programme (grant agreement No. 803711).

	\Addresses
	

\begin{thebibliography}{999}
		
		\bibitem{BDV ref} V. Beresnevich, D. Dickinson, S. Velani,
		\emph{Measure theoretic laws for lim sup sets}, Mem. Amer. Math. Soc. \textbf{179}(846), (2006).
		
		
		\bibitem{BHV}V. Beresnevich, A. Haynes and S. Velani, \emph{Sums of reciprocals of fractional parts
			and multiplicative Diophantine approximation}, Mem. Amer. Math. Soc., \textbf{263}(1276), 2020.
		
		\bibitem{C18} S. Chow, \emph{Bohr sets and multiplicative Diophantine approximation},  Duke Math. J. \textbf{167}(9), (2018), 1623-1642.
		
		\bibitem{CT} S. Chow and N. Technau, \emph{Littlewood and Duffin-Schaeffer-type
			problem in Diophantine approximation}, to appear in Mem. Amer. Math. Soc.,arXiv:2010.09069
		
		\bibitem{DS} R. Duffin and A. Schaeffer, \emph{ Khintchine’s problem in metric Diophantine approximation}, Duke Math. J. \textbf{8} ,(1941), 243–255.
		
		\bibitem{DT97} M. Drmota and R. Tichy, \emph{ Sequences, Discrepancies and Applications}, Lecture Notes in Mathematics,
		Springer-Verlag Berlin Heidelberg,(1997).
		
		
		\bibitem{Gallagher} P. Gallagher, \emph{Approximation by reduced fractions}, J. Math. Soc. Japan \textbf{13}(4), (1961), 342–345.
		
		\bibitem{Gallagher62} P. Gallagher, \emph{Metric Simultaneous Diophantine Approximation}, J. Lond. Math. Soc. \textbf{37}{1}, (1962), 387-390.
		
		\bibitem{KM2019} D. Koukoulopoulos and J. Maynard, \emph{On the Duffin--Schaeffer conjecture}, Ann. of Math., \textbf{192}, 2020, 251-307.
		
		
		
		\bibitem{Ramirez} F. Ram\'{i}rez, \emph{Counterexamples, covering systems, and zero-one laws for inhomogeneous approximation}, Int. J. Number Theory, \textbf{13}(3), (2017), 633-654. 
		
		
		\bibitem{Szusz} P. Sz\"{u}sz, \emph{\"{U}ber die metrische Theorie der Diophantischen Approximation}, Acta. Math. Sci. Hungar. \textbf{9}, (1958) ,177-193.
		
		
		\bibitem{Yu} H. Yu, \emph{A Fourier-analytic approach to inhomogeneous Diophantine approximation}, Acta Arith. \textbf{190}, (2019), 263-292.
		
		\bibitem{Yu2} H. Yu, \emph{On the metric theory of inhomogeneous Diophantine approximation: An Erd\H{o}s-Vaaler type result}, Journal of Number Theory, \textbf{224}, (2021), 243-273
		
	\end{thebibliography}
\end{document}